\documentclass[12pt,reqno]{amsart}
\usepackage{amsmath, amssymb, amsthm, esint, verbatim, hyperref}
\usepackage{times}
\usepackage{dsfont}

\usepackage{enumerate}

\usepackage{wrapfig}
\usepackage{tikz}
\usetikzlibrary{decorations.fractals}

\numberwithin{equation}{section}

\hypersetup{
  pdftitle={Calibrated Reifenberg theorem with holes},
  pdfauthor={Susanna Bertolini, Alessandro Preti, Daniele Valtorta},
  pdfsubject={},
  pdfkeywords={},
  pdfpagelayout=SinglePage,
  pdfpagemode=UseOutlines,
  colorlinks,
  bookmarksopen,
  linkcolor=[rgb]{0,0,0.7},
  urlcolor=[rgb]{0,0,0.4},
  citecolor=[rgb]{0.4,0.1,0},
}

\newtheorem{theorem}{Theorem}[section]

\newtheorem{proposition}[theorem]{Proposition}

\newtheorem{lemma}[theorem]{Lemma}
\newtheorem{corollary}[theorem]{Corollary}
\theoremstyle{definition}
\newtheorem{definition}[theorem]{Definition}
\theoremstyle{remark}
\newtheorem{remark}[theorem]{Remark}
\theoremstyle{remark}
\newtheorem{example}[theorem]{Example}
\theoremstyle{remark}

\theoremstyle{remark}

\theoremstyle{remark}

\newcommand{\dist}{\mathrm{dist}}

\newcommand{\cH}{\mathcal{H}}

\newcommand{\R}{\mathds{R}}

\newcommand{\dR}{\mathds{R}}

\newcommand{\norm}[1]{\left\|#1\right\|}
\newcommand{\ps}[2]{\left\langle#1,#2\right\rangle}
\newcommand{\ton}[1]{\left(#1\right)}
\newcommand{\qua}[1]{\left[#1\right]}
\newcommand{\cur}[1]{\left\{#1\right\}}
\newcommand{\abs}[1]{\left|#1\right|}

\newcommand{\B}[2]{B_{#1}\ton{#2}}

\newcommand{\hol}{H\"older }

\title{Calibrated Reifenberg with holes}
\author{Susanna Bertolini} \author{Alessandro Preti} \author{Daniele Valtorta}
\date{\today}
\thanks{S. B. was supported by the European Research Council under the Grant Agreement No 948029}

\begin{document}
\begin{abstract}
 In this article, we study a calibrated version of Reifenberg theorem ``with holes''. In particular we study sets that are suitably approximable at all points and scales by calibrated planes and show that, without any additional hypotheses on $\beta$-numbers, this implies measure upper bounds and rectifiability. This article follows the main techniques introduced in \cite{ENV_calibrated_Reifenberg}, but it allows for holes in the sets under consideration, and is more self-contained.
\end{abstract}
\maketitle
\tableofcontents

\section{Introduction}

A celebrated theorem by Reifenberg \cite{reif_orig} (see also \cite{simon_squash}) states that a if subset $S\subset \R^n$ is sufficiently close to $k$-dimensional plane at all points and scales, then this set is $C^{0,\alpha}$-equivalent to a $k$-dimensional flat plane. In particular, we have
\begin{theorem}[Reifenberg's theorem]\label{th_Reif_standard}
For all $\alpha\in [0,1)$, there exists $\delta=\delta(n,\alpha)$ such that the following holds. Let $S\subset \R^n$ be a closed set such that for all $x\in \B 2 0$ and $r\leq 2$:
\begin{gather}
 \inf_{L_{x,r}}\cur{d_\cH(S\cap \B r x,L_{x,r}\cap \B r x)\, s.t. \ \ L_{x,r} \ \ \text{is a $k$-dimensional affine plane}}\leq \delta r\, ,
\end{gather}
then there exists a $C^{0,\alpha}$ map $\phi:\R^k\to \R^n$ satisfying
\begin{gather}
 S\cap \B 1 0 \subset \phi(\B {1+C(n)\delta }{0})\subset S\cap \B {1+2C(n)\delta} 0\, .
\end{gather}
Moreover, $\phi$ is $C^{0,\alpha}$ bi-\hol onto its image.
\end{theorem}

Easy examples show that this theorem as stated cannot be improved to show that $S$ has $k$-dimensional volume bounds, as explained in Section \ref{sec_beta}. Indeed, additional assumptions must be made for this to hold.

Various works have improved on Reifenberg's original theorem, showing a Lipschitz equivalence, or simply effective measure bounds assuming some Dini summability condition on Jones' $\beta$-numbers for the set $S$. A brief overview of these works is present in Section \ref{sec_beta}.

However, other assumptions can be made on $S$ in order to improve on Reifenberg's original theorem, and in this article we follow the strategy of \cite{ENV_calibrated_Reifenberg} and deal with the almost calibrated case. In particular, we show that given an \textit{almost calibration} $\Omega$ and a set $S\subset \R^n$ that satisfies
 \begin{gather}\label{semireif}
 	S\cap B_r(x)\subset \B{\delta r}{L_{x,r}} \, ,
 \end{gather}
where $L_{x,r}$ is a $k$-dimensional affine plane almost calibrated by $\Omega$, then $S$ is $k$-rectifiable and has uniform upper $k$-dimensional measure bounds.

As opposed to \cite{ENV_calibrated_Reifenberg}, we do not require the set $S$ to satisfy a two-sided condition of the form
 \begin{gather}\label{double-sided-reif}
 	d_\cH(S\cap B_r(x),L_{x,r}\cap \B r x)\leq \delta r \, .
 \end{gather}
For this reason, this result is more general, and can be applied also to sets ``with holes''. As an illustration of this, consider that any set $S$ contained in a $k$-dimensional plane clearly satisfies \eqref{semireif}, while \eqref{double-sided-reif} requires the closure of $S$ to be equal to $L$.

The strategy of the proof is based on an adaptation of the techniques in \cite{ENV_calibrated_Reifenberg}, along with an inductive covering argument.

\section{Main theorem and applications}
We start by recalling the definition of almost calibration from \cite{ENV_calibrated_Reifenberg}, necessary to state the main result. The idea behind it similar to a geometric calibration, i.e. a closed $k$-form $\Omega$ such that $\Omega|_L\leq \operatorname{vol}_L$ for all $k$-dimensional oriented subspaces. Here we relax the conditions in the definition of calibration in a quantitative way.

\begin{definition}[$\eta$-Calibration]
 Let $\Omega$ be a smooth $k$-form over $B_2(0)\subseteq \R^m$. We say that $\Omega$ is an $\eta$-calibration if
 \begin{enumerate}
  \item $|\Omega-\Omega_0|\leq \eta$ for a constant form $\Omega_0$ ,
  \item for all $x\in \R^m$ and any oriented $k$-dimensional subspace $L\subseteq \R^m$, we have $\Omega[L]\leq 1+\eta$\, .
 \end{enumerate}
\end{definition}
\begin{remark}
	If $L^k$ is an oriented subspace then we define $\Omega[L]=\Omega[e_1,\ldots,e_k]$ where $e_1,\ldots e_k$ is any oriented orthonormal basis of $L$.
\end{remark}

Another key definition for the purpose of proof will be the concept of $\varepsilon$-independence, which is a quantitative version of the notion of linear independence. As we will see in Lemma \ref{lemma_quantitative_GS}, this notion is stable under small perturbations.

\begin{definition}[$\varepsilon-$Linear Independence]\label{def:epsbad}
We say that $\cur{e_i}_0^k\in B_r$ is a set of $\varepsilon-$linear independent vectors at scale $r$ if for each $i$
\begin{equation}
  e_{i+1}\notin B_{\varepsilon r}\ton{e_0 +\textrm{span}\ton{e_i-e_0, ..., e_i-e_0}}.
\end{equation}
\end{definition}

Our main result is that if our $\delta$-Reifenberg surface is uniformly positive with respect to an almost calibration $\Omega$, then it must be rectifiable with Ahlfor's regularity estimates:

\begin{theorem}[Rectifiable Reifenberg for Almost Calibrations]\label{t:main}
 Let $S\subset B_2(0)\subseteq \R^n$ be a closed set with $0\in S$ and let $\Omega$ be an $\eta$-calibration. Then for all $2\eta<\alpha<1$ $\exists$ $\delta(n,\varepsilon, \eta)>0$ and $ A(n,\alpha,\eta)>1$ such that if for all $B_{r}(x)\subseteq \B 2 0$ there exists an oriented $k$-dimensional subspace $L=L_{x,r}$ such that
 \begin{gather}
 	S\cap B_r(x)\subset \B{\delta r}{L_{x,r}}  \, ,\;\;\;\Omega[L_{x,r}]>\alpha>0\, ,
 \end{gather}
 then $S\cap \B 1 0$ is $k$-rectifiable and for all $x\in S$ with $B_{2r}(x)\subseteq \B 2 0$ we have that
\begin{gather}\label{me}
 \frac{\cH^k(S\cap \B r x)}{\omega_k r^k}\leq A\, .
\end{gather}
\end{theorem}
\begin{remark}
 It is worth mentioning that, as opposed to the results in \cite{ENV_calibrated_Reifenberg}, here we do not require a two-sided Reifenberg condition on the set $S$. This allows $S$ to have ``holes'' in it, and in turn rules out any possible lower bound on the $k$-dimensional measure of $S$.
\end{remark}

\subsection{\texorpdfstring{$\beta$}{beta}-number approach}\label{sec_beta}
Although this article deals with calibrated versions of Reifenberg's theorem, it is worth discussing another more standard approach to generalize Reifenberg's theorem involving $\beta$-numbers. We start by recalling a standard counterexample to the bi-Lipschitz version of Reifenberg's theorem.

By simple examples, it is clear the Reifenberg flat condition does not imply finiteness of the $k$ dimensional Hausdorff measure of $S$, or its rectifiability. A classical example of this (see \cite{davidtoro}, \cite[section 4.13]{mattila} or \cite[Example 5.2]{ENV_HB}) is the snowflake construction.

\vspace{5mm}

 \begin{wrapfigure}{r}{45mm}
\begin{center}
 \begin{tikzpicture}[]
  \draw decorate{
       (0,0) - +(0:4)  };
\end{tikzpicture}

\vspace{2mm}
\begin{tikzpicture}[decoration=Koch snowflake]
  \draw decorate{
       (0,0) -- ++(0:4)  };
\end{tikzpicture}
\vspace{2mm}

\begin{tikzpicture}[decoration=Koch snowflake]
  \draw decorate{ decorate{
       (0,0) -- ++(0:4)  }};
\end{tikzpicture}
\vspace{2mm}

\begin{tikzpicture}[decoration=Koch snowflake]
  \draw decorate{decorate{ decorate{
       (0,0) -- ++(0:4)  }}};
\end{tikzpicture}
\end{center}
\end{wrapfigure}

In the Euclidean $\R^2$, consider the one dimensional segment $I_0=[0,1]\times\{0\}$, and replace its middle section $M_0=[1/3,2/3]\times \{0\}$ with the upper part of the isosceles triangle with base $M_0$ and height $\eta \cH^1(M_0)$, with $\eta\leq \sqrt 3 /2$. The resulting curve is denoted by $I_1$. By induction, we can repeat this construction on each of the straight segments of $I_1$ to obtain $I_2$, and so on. It is clear from the construction that the Hausdorff distance between $I_k$ and $I_{k+1}$ is smaller than $3^{-k}\eta$, and thus we can define a limit $I_\infty$. On the right we sketch the construction for the first few steps of the standard snowflake (with parameter $\eta=\sqrt 3 /2$).

It is easy to see that all of the sets $I_\infty(\eta)$ satisfy the assumptions of Reifenberg's original theorem \ref{th_Reif_standard} with parameter $c\eta$, and it is equally clear that if $\eta>0$, then the length $\cH^1\ton{I_\infty(\eta)}=\infty$. This shows that, without additional assumptions, Reifenberg's original theorem \ref{th_Reif_standard} cannot be pushed to obtain a bi-Lipschitz equivalence.

\vspace{5mm}
An interesting observation is that if in the snowflake construction we replace the fixed parameter $\eta$ with a variable parameter $\eta_k$, then the length of $I_\infty(\cur{\eta_k)}$ is finite if and only if $\sum \eta_k^2$ is finite.

A nice way to generalize this example to a theorem is to exploit Jones' $\beta$-numbers, which are a quantitative notion of how close a set is to an affine plane at different points and scales. In particular, one can define
\begin{gather}
 \beta_\infty(x,r) = \inf_{L } \cur{\frac{d_\cH(S\cap \B r x, L\cap \B r x )}{r} \, , \ \ s.t. \ \ L \text{ is an affine plane}}\, ,
\end{gather}

Various adaptation of Reifenberg's theorem that assume some Dini summability for the $\beta$ numbers have been studied in literature. The first example is  \cite{toro:reifenberg}, where the author assume the pointwise summability condition $\int_0^2 \beta_\infty(x,r)^2 \frac{dr}{r}<\infty$ for all $x\in S\cap \B 2 0$ in order to prove that $S\cap \B 2 0$ is in a bi-Lipschitz correspondence with a flat plane. Similar results with quantitative estimates are available in \cite{ENV}.

\vspace{5mm}

\section{Proof of the main theorem}
In this section we prove our main Theorem \ref{t:main}, by splitting it into three main steps:
\begin{enumerate}
 \item given our $\delta$-Reifenberg flat set $S$, we produce a covering Lemma that splits it up into two main pieces:
  \begin{gather}\label{apporx}
   S\cap \B 1 0 \subset S_0 \bigcup \bigcup_{i\in I_0}\B {r_i}{x_i}
  \end{gather}
 where $S_0$ is rectifiable with suitable estimates on the $k$-dimensional measure of $S_0$. 
 The balls $\B {r_i}{x_i}$ will be chosen very carefully through a good ball/bad ball corona-type decomposition, in order to obtain a bound on $\sum r_i^k$ and to be able to reapply the covering Lemma inductively without losing control on the overall measure estimates.
 \item the proof of the covering Lemma, which is the most technical part of the article. This will be obtained by applying the Reifenberg's calibrated construction of \cite{ENV_calibrated_Reifenberg}. Here we carry out the proof with all details, and in particular we adapt to the calibrated case the construction of the approximating manifolds of \cite{Naber_Park} to the calibrated case.
 \item An inductive application of the covering Lemma will yield the desired results in the main theorem.
\end{enumerate}

We start by stating the definition of good balls and bad balls, and the covering Lemma.

\subsection{Good balls and bad balls}
Given a set $S\subset \R^n$ satisfying the assumptions of the main Theorem \ref{t:main}, and $\B r x\subset \R^n$, we say that $\B r x$ is a good ball for $S$ if $S\cap \B r x$ is effectively spanning a dimension $k$ subspace, and a bad ball otherwise. Recall that, roughly speaking, $S\cap \B r x$ is always contained in a $\delta r$ neighborhood of some $k$ dimensional plane $L_{x,r}$. Thus if $\B r x$ is a bad ball, then $S\cap \B r x$ is contained in a small neighborhood of a smaller, $k-1$ dimensional, subspace.

The relevance of this definition will become clear in Lemma \ref{lemma_planes}, where we will see that if $\B r x$ is a good ball, then the approximating planes for $S$ on balls $\B s y$ close enough to $\B r x$ (meaning, with nearby centers and comparable radia) are close to each other.

\vspace{5mm}
In order to make these definitions more precise, we start with the definition of $\varepsilon$-Linear Independence for subsets of the manifold $S$, which will make use of Definition \ref{def:epsbad}.

\par{
  \begin{definition}[$\varepsilon-$Linear Independence]\label{def:bad}
    We say a set $S\subset B_r$ is $(k,\varepsilon)-$Linearly Independent if it contains a $\varepsilon-$Linearly Independent set of points $\cur{e_i}_0^k\in S$.
  \end{definition}
  This means that even though $S$ may have holes, $S$ still contains enough points to effectively span a $k$-dimensional subspace. Notice that
  \begin{remark}
    If $S\subset B_r$ is not $(k,\varepsilon)-$Linearly Independent, then it is contained in a tubular neighborhood of width $\varepsilon r$ of a $(k-1)-$dimensional affine space.
  \end{remark}

  The set of balls $\B r x$ for which $S\cap \B r x$ is not is not $(k,\varepsilon)-$linearly independent will be our set of ``bad balls'', the others will be the ``good balls''.

  A standard lemma regarding $\varepsilon$-linearly independent sets is that these sets can be used as a basis with quantitative estimates.
  \begin{lemma}\label{lemma_quantitative_GS}
   If $\cur{e_i}_{i=0}^k$ are $\varepsilon$-linear independent in $\B r x$, then for all
   \begin{gather}
    p\in V=p_0+\operatorname{span}\cur{e_1-e_0,\cdots,e_k-e_0}
   \end{gather}
   there exist $\cur{\lambda_i}_{i=1}^k\in \R^k$ such that
  \begin{gather}
   p=e_0+\sum_{i=0}^k \lambda_i (e_i-e_0)\, , \qquad |\lambda_i|\le c(n,\varepsilon)\norm{p-e_0}\, .
  \end{gather}
  \end{lemma}
  \begin{proof}
    This lemma is quite standard, and it can be found for example in \cite[Lemma 3.13]{ENV_HB}. The proof is a simple application of the Gram-Schmidt orthonormalization procedure. We can assume WLOG that $x=0$ and $r=1$.

    Given that $\cur{e_i-e_0}_{i=1}^k$ are linearly independent vectors, we can apply Gram-Schmidt to obtain an orthonormal basis $\hat e_1,\cdots,\hat e_k$. By induction, one can easily show that for all $j$:
  \begin{gather}
   \hat e_j = \sum_{i=1}^j \lambda_i (e_i-e_0)\, \qquad \text{with}\qquad |\lambda_i|\leq c(n,\varepsilon)\, .
  \end{gather}
  We start from the base case $j=1$: then, following the Gram-Schmidt procedure, we define
  \begin{equation}
      \hat{e}_1 = \frac{e_1 - e_0}{\abs{e_1 - e_0}} = \lambda_1\ton{e_1 - e_0},
  \end{equation}
  with $\lambda_1 = \abs{e_1 - e_0}^{-1}$. As the set $\cur{e_0, e_1}$ is $\varepsilon-$Linear Independent, by definition we have $e_1\notin B_{\varepsilon}\ton{e_0}$, and thus $\abs{e_1 - e_0}\geq \varepsilon$. This gives the desired estimate on $\lambda_1$. \newline
  We now assume the estimate holds up to $j-1$ and we prove it also holds for $j$. Then, again by the Gram-Schmidt orthogonalisation process we obtain that 
  \begin{equation}
      \hat{e}_{j} = \frac{\ton{e_{j}-e_0} - P_{j-1}\ton{e_{j}-e_0}}{\abs{\ton{e_{j}-e_0} - P_{j-1}\ton{e_{j}-e_0}}}, 
  \end{equation}
  where $P_{j-1}$ is the projection on the subspace generated by $\cur{\hat{e}_1, ..., \hat{e}_{j-1}}$. Now, by the induction hypothesis applied to $p = e_0 + P_{j-1}(e_{j}-e_0)\in e_0 + \textrm{span}\cur{e_1 -e_0, ..., e_{j-1} - e_0}$ we have 
  \begin{equation}
      P_{j-1}(e_{j} - e_0) = \sum_{i=1}^{j-1} \mu_i \ton{e_i - e_0}, 
  \end{equation}
  with $\abs{\mu_i}\leq c_j(n,\varepsilon)$. Moreover, in the same way as in the base case we conclude that by the definition of $\varepsilon-$Linear Independence $\abs{\ton{e_{j}-e_0} - P_{j-1}\ton{e_{j}-e_0}}\geq \varepsilon$. From here, we conclude as in the base case. 
  \\
  Now the estimate follows from the expansion of $p-e_0$ in the orthonormal basis $\cur{\hat e_i}$, and the fact that even though in every step of the induction process we get a new (possibly bigger) constant, as $k\leq n$ it is sufficient to choose $c = \max_k c(k,n,\varepsilon)$, which will depend on $\varepsilon$ and $n$ only. 
  \end{proof}
  As a corollary, we can prove that the condition of being $\varepsilon$-linearly independent is stable under ``small movements'' in the vectors.
  \begin{corollary}\label{cor_epsmezzi}
   For all $\varepsilon>0$, there exists $\delta_0(n,\varepsilon)>0$ such that if $\cur{e_i}_{i=0}^k$ are $\varepsilon$-linear independent in $\B r x$, and
   \begin{gather}
    \abs{f_i-e_i}\leq \delta_0(n,\varepsilon)
   \end{gather}
   then $\cur{f_i}_{i=0}^k$ are $\varepsilon/2$ linearly independent in $\B r x$.
   \end{corollary}
   \begin{proof}
    We can assume WLOG that $x=0$ and $r=1$. We prove this by induction on $j=1,\cdots,k$.

    For $j=1$, this is a simple application of the triangle inequality. Suppose now that by induction the lemma is proved up to $j-1$, so that $\cur{f_i}_{i=0}^{j-1}$ are $\varepsilon/2$-linearly independent. Set for convenience
    \begin{gather}
     e_i-e_0=v_i\, , \qquad f_i-f_0=w_i\, .
    \end{gather}
    We need to show that
    \begin{gather}
     d(w_j,\operatorname{span}(w_1,\cdots,w_{j-1}))\geq \varepsilon/2\, .
    \end{gather}
    Let $\hat w_j$ be the projection of $w_j$ onto $\operatorname{span}(w_1,\cdots,w_{j-1})$. By the previous Lemma, we have
    \begin{gather}
     d(w_j,\operatorname{span}(w_1,\cdots,w_{j-1}))=\abs{w_j-\hat w_j}=\abs{w_j-\sum_{i=1}^{j-1}\lambda_i w_i}
    \end{gather}
    with $\abs{\lambda_i}\leq C(n,\varepsilon)$. Then by the triangular inequality:
    \begin{gather}
     \abs{w_j-\sum_{i=1}^{j-1}\lambda_i w_i}\geq \abs{v_j-\sum_{i=1}^{j-1}\lambda_i v_i}-\abs{w_j-v_j}-\sum_{i=1}^{j-1}\abs{\lambda_i} \abs{w_i-v_i}\geq \varepsilon-C(n,\varepsilon)\delta_0\, .
    \end{gather}
    This concludes the proof.

   \end{proof}

  \subsection{Tilting control for the approximating planes between good balls}
  In this section we develop some tools to control the tilting of approximating planes $L_{x,r}$ between good balls of comparable size. We start with an example that shows that this control cannot be obtained on bad balls.
  \begin{example}
   As an easy example of a bad ball, consider the set  
   \begin{gather}
    S=\ton{[-1,-1/2]\cup \cur{0}\cup [1/2,1]}\times\cur{0}\subset \R^2
   \end{gather} It is clear that this set is a one dimensional Reifenberg flat set, being a subset of a straight line.
   Moreover, if we consider $\B r 0$ with $r\leq 1/2$, then clearly $S\cap \B r 0=\cur{0}$.
   This shows that any one dimensional line going through the origin is an approximating line for $S$ on $\B r 0$, and thus it is not possible to bound in an effective way the distance between an approximating line at scale $1$ and an approximating line at scale $1/3$.
  \end{example}

  \begin{definition}\label{good_B}
      We say that $B_r(x)$ is a \textit{good ball} if $B_r(x)\cap S$ is $(k,\varepsilon)$ linear independent.
  \end{definition}
  As per the example above, in general we cannot prove that if two balls have centers close enough and comparable radii then the approximating subspaces are quantitatively close. However, if the two balls are good balls then the result is still true, as shown in the next lemma.

  \begin{lemma}\label{lemma_planes}
    Let $S$ satisfy the assumptions of the main Theorem \ref{t:main}, and fix $x,y$ satisfying the following conditions:
    \begin{itemize}
        \item $B_r(x)\subset B_s(y)$, with $r\geq \frac{1}{10^4}s$;
        \item $B_r(x)\cap S, B_s(y)\cap S$ are $(k, \varepsilon)-$linearly independent;
    \end{itemize}
    then $d_\mathcal{H}\ton{L_{x, r}\cap B_s(y), L_{y, r}\cap B_s(y)}\leq C(n, \varepsilon)\delta r$. 
  \end{lemma}

  \begin{proof}
    \par{
    By scaling and translating, we can assume WLOG that $r=1$ and $x=0$.
    Let $p_0, ..., p_k\in S\cap B_1(0)$ be a set of $\varepsilon-$Linearly Independent points, which exist by definition of good ball.

    Then, by the definitions of both $L_{0,1}$ and $L_{y, s}$, we can find
    \begin{enumerate}[i.]
        \item $x_0, ..., x_k\in L_{0,1}\cap B_1(0)$ such that $\abs{x_i - p_i}\leq \delta $;
        \item $y_0, ..., y_k\in L_{y,s}\cap B_1$ such that $\abs{y_i - p_i}\leq \delta s\leq C\delta $.
    \end{enumerate}
    }
    By Corollary \ref{cor_epsmezzi}, both the sets $\cur{y_0, ..., y_k}$ and $\cur{x_0, ..., x_k}$ are $\varepsilon/2$ independent.
    \par{
    We can now calculate the Hausdorff distance between the planes. Take any $z\in L_{0,1}\cap B_1$. Then, as the set $\cur{v_i = x_i - x_0}$ is a base for $L_{0,1}-x_0$ we can write $x-x_0 = \sum_j \lambda_j v_j$ with $\abs{\lambda_j}\leq C(n,\varepsilon)$. Then, take $w_j = y_j -y_0$ and consider $t = y_0 + \sum_j \lambda_j w_j$. Then,
    \begin{equation}
        \abs{z-t} \leq \abs{x_0 - y_0} +  \sum_j\abs{\lambda_j}\abs{v_j-w_j} \leq C(n, \varepsilon)\delta\, .
    \end{equation}
    }
    \par{
    The other direction can be proved in a similar way, and this concludes the proof.
    }
  \end{proof}

\begin{remark}
    Note that having $\varepsilon-$linear independence in this case is crucial, as it provides an upper bound for $\abs{\lambda_j}$. By asking for $\varepsilon-$linear independence instead of just linear independence in the definition of good balls, we are ensuring that when we find an orthonormal basis we have a lower bound on the norm of its elements (as in Lemma \ref{lemma_quantitative_GS}), and therefore that our coefficients do not explode. 
\end{remark}

\subsection{Main covering lemma}

We now show the existence of an approximating manifold $S_r$, relative to $S$.  The rough idea is that the family of smooth manifolds $S_r$ will approximate $S$ at scale $r$ on good balls. The manifolds $S_r$ will inherit the almost calibration properties of $S$, and we will exploit this to prove that these manifolds have uniform measure bounds.

To be more precise, we have:

\begin{lemma}\label{lemma_covering}
 Given a set $S$ satisfying the assumptions of Theorem \ref{t:main}, there exists a one parameter family $S_r$ of manifolds and a family of ``bad'' balls $\cur{\B {r_{x_i}}{x_i}}_{i\in I_b}$ such that
 \begin{enumerate}
  \item $x_i\in B_{\varepsilon r_{x_i}}(S_{r})$ for all $r\leq r_{x_i}$
  \item $S_r\cap \B {\frac{r_{x_i}}{20}}{x_i}$ is independent of $r$ if $r\leq \frac{r_{x_i}}{2}$
  \item $\B {\frac{r_{x_i}}{5}}{x_i}\cap \B {\frac{r_{x_j}}{5}}{x_j}=\emptyset$ if $i\neq j$
  \item we have
  \begin{gather}
   S\cap \B 1 0 \subset S_0 \cup \left(\bigcup_{i\in I_b}\B {r_{x_i}}{x_i}\right)
  \end{gather}
  \item locally at scale $r$, $S_r$ is a Lipschitz graph over some $k$ dimensional plane with Lipschitz constant $C \delta$.
  \item $S_r$ are almost calibrated, in the sense that for all $x\in S_r$:
  \begin{gather}
   \Omega[T_x(S_r)]\geq \frac{\alpha}{4} >0
  \end{gather}
  \item\label{it_Vbad} $\forall i\in I_b$, there exists a $k-1$ dimensional subspace $V_{x_i}$ such that
  \begin{gather}
   S\cap \B{r_{x_i}}{x_i}\subset \B {\varepsilon r_{x_i}}{V_{x_i}}
  \end{gather}
 \end{enumerate}
\end{lemma}

A thorough description on the creation of the approximating manifold $S_r$ will follow in the paper. This construction is based on the same ideas as in \cite[Theorem 4.2]{Naber_Park}, adapted to a setting where the manifold $S$ may have holes. This will be done in a few steps:
\begin{enumerate}
    \item Construction of a radius function \( r_x : B_2(0) \to \mathbb{R} \), interpreted as the admissible scale at which well-behaved approximating planes associated to each point \( x \in S \) exist. This function accounts for the portions of the manifold \( S \) that lie in bad balls, effectively fixing the scale.
    \item Construction of a partition of unity argument for the set $B_2(0)$ with controlled derivatives.
    \item Construction of a map that assigns to each point in \( B_2(0) \) an approximating plane to \( S \), built using the radius function \( r_x \) and a partition of unity argument, followed by an analysis of the map’s smoothness properties.
    \item Definition of the approximating manifold \( S_r \) and proof of Lemma~\ref{lemma_covering}.
\end{enumerate}

We therefore begin by defining the function \( r_x \) via an auxiliary function \( s_x \), employing the Vitali covering lemma and then extending \( r_x \) to all of \( B_2(0) \). This will also be very useful to define the set of bad and good balls relative to the covering.

\begin{definition}\label{def1}
    Let \( S \subseteq B_2 \subseteq \mathbb{R}^n \) satisfy condition \eqref{semireif} for $\delta > 0$. We define the real valued function $s_x: S \rightarrow \mathbb{R}^+$
\begin{equation}
    s_x = \inf \left\{ s \in [r_0, 1] \;\middle|\; \forall t \in [s, 1], \; B_{t}(x) \text{ is a good ball} \right\}.
\end{equation}
\end{definition}

\begin{remark}
    By definition, we have \( S \subset \bigcup_{x \in S} \overline{B_{\frac{s_x}{5}}(x)} \)\, .
\end{remark}
We use the function $s_x$ to construct a covering. Since \(4 \geq s_x \geq r_0 \) uniformly for all \( x \in S \), we may apply Vitali’s covering lemma to the family \( \{ B_{\frac{s_x}{5}}(x) \}_{x \in S} \). This yields a finite subset \( C \subset S \) such that
\[
S \subset \bigcup_{x \in C} B_{s_x}(x),
\]
and the collection \( \{ B_{\frac{s_x}{5}}(x) \}_{x \in C} \) is pairwise disjoint. This covering allows us to define the collection of \emph{bad balls} that cover \( S \), which will be treated using a recursive argument. Indeed, whenever \( s_x \neq r_0 \), the function is indicating that \( x \) lies near a region of \( S \) that does not satisfy Definition~\ref{good_B}.

\begin{definition}\label{def:bad-balls}
    The collection of \emph{bad balls} associated to the manifold \( S \) and the parameter \( r_0 \) is defined as the family of open balls \( \mathcal{B} = \{ B_{r_i}(x_i) \}_{x_i \in C} \), where each \( r_i = s_{x_i} > r_0 \). The balls of the covering satisfying the equality are consequently named \emph{good balls} relative to $S$ and $r_0$.
\end{definition}

\begin{remark}
    This new notion of “good” and “bad” should be understood as an extension of Definition \ref{good_B}, formulated relative to a Vitali covering of the space $S$. From now on, any reference to a family of good or bad balls will always be with respect to this collection.
\end{remark}

We can now continue with the definition of the function $r_x$ obtained extending the function \( s_x \) on the entire ball \( B_2(0) \).

\begin{definition}\label{def2}
    Given \( S \), \( C \), and \( s_x: C \rightarrow \mathbb{R}^+ \), we define \( r_x: B_2(0) \rightarrow \mathbb{R}^+ \) as:
    \begin{equation}
        r_x := \sup \left\{ 0 < s < 2 : \forall y \in C \cap B_{s/5}(x), \ s_y \geq s \right\}.
    \end{equation}
\end{definition}

Note that this definition is in some sense ``solid'', as it is indeed an extension of the previous one; moreover, it has good regularity qualities. 

\begin{theorem}\label{rthm}
    Given the function $r_x$ defined above the following is true:
    \begin{enumerate}
        \item $s_x = r_x$ \quad $\forall x \in C$.
        \item $r_x \geq  d(x,C) \geq d(x,S)$. 
        \item $r_x$ is Lipschitz with $\text{Lip}(r_x) \leq 5$.
    \end{enumerate}
\end{theorem}

\begin{proof}
    The point $(\textit{1})$ can be proven using a property of the Vitali covering: given $x \in C$, it follows $B_{\frac{s}{5}}(x) \cap C = \{x\}$ for every $s \leq s_x$. Then, $s_x - \frac{1}{n}\in \cur{0<s<2 : \forall y \in C\cap B_{s/5}(x), \; s_x\geq s}$, and therefore $r_x \geq \sup_{n} s_x - \frac{1}{n} = s_x$.  By definition, now we have $s_x \geq r_x \geq s$ for every $s \in [0,s_x]$. This proves $s_x = r_x$ if $x\in C$. 
    To prove $(\textit{2})$ it is enough to use the definition: indeed if we consider the ball $B_{s/5}(x)$ with $s < 5d(x,C)$, then $C \cap B_{s/5}(x) = {\emptyset}$ and thus the condition over $s$ is automatically satisfied; as $C\subset S$ we also know that $d(x,C) \geq d(x,S)$, hence the second inequality. 
    
    Consider now $x, z \in B_2$ and assume for simplicity that $r_x \geq r_z$. Set $\gamma = 5|x-z|$ and consider every $s \in (\gamma,2)$ for which it holds that $\forall y \in C \cap B_{\frac{s}{5}}(x)$ it holds $s_y \geq s$. Note that we can always assume there is at least one such value, as if this is not the case then $r_z\leq r_x\leq \gamma$, and there is nothing to prove. Let us now consider the value $s' = s- \gamma$ and notice that $C \cap B_{s'/5}(z) \subset C \cap B_{s/5}(x)$ which leads to $s_y \geq s \geq s'$.
    Since for every feasible value $s$ for $r_x$ $s -\gamma$ belongs to the set of feasible real values for $r_z$ then by taking the supremum, as prescribed in the definition, allows us to conclude that $r_z \geq r_x - \gamma$ and thus $r_x - r_z \leq 5|x-z|$, which is the estimate needed to prove $(\textit{3})$. 
  
\end{proof}

\begin{remark}
    The second point will be useful later to guarantee a non-empty intersection between $B_{r_x}(x)$ and $S$.
\end{remark}

We now proceed to establish a \emph{partition of unity} argument required to build the family of approximating manifolds $S_r$ with good properties.

To this end, we define a modified radius function that incorporates an arbitrary scale parameter \( r > 0 \). Specifically, we set
\begin{gather}
    \tilde{r}_x : B_2(0) \rightarrow \mathbb{R}^+, \quad \tilde{r}_x := \frac{r_x \vee r}{100}\, ,
\end{gather}
where \( r_x \vee r =\max\cur{r_x,r}\) denotes the maximum of \( r_x \) and \( r \).

To justify the construction, we first derive a Lipschitz estimate for \( \tilde{r}_x \), which ensures that its values are comparable at nearby points; that is, for any two sufficiently close points, the corresponding values of \( \tilde{r}_x \) remain within a controlled ratio.

\begin{lemma}\label{mybound}
    Let us consider two points $x_\alpha,x \in B_2$ such that $B_{k \tilde{r}_{x_\alpha}}(x_\alpha) \cap B_{k\tilde{r}_{x}}(x) \neq \emptyset$ for a given real number $k < 20$, then there exists a positive real value $w$ such that:
    \begin{equation}
       \frac{1}{w}\tilde{r}_{x_\alpha} \leq  \tilde{r}_x \leq  w\tilde{r}_{x_\alpha}
    \end{equation}
\end{lemma}

\begin{proof}
    By point \textit{(3)} of Theorem~\ref{rthm}, we immediately obtain the Lipschitz bound
    \begin{gather}
    \text{Lip}(\tilde{r}_x) \leq \frac{1}{20}.
    \end{gather}
    Assume without loss of generality that \( \tilde{r}_x \geq \tilde{r}_{x_\alpha} \). Using this assumption and the fact that the balls \( B_{k \tilde{r}_{x_\alpha}}(x_\alpha) \) and \( B_{k \tilde{r}_x}(x) \) intersect (i.e., \( B_{k \tilde{r}_{x_\alpha}} \cap B_{k \tilde{r}_x} \neq \emptyset \)), we deduce
    \begin{equation}\label{lassu}
        |\tilde{r}_x - \tilde{r}_{x_\alpha}| = \tilde{r}_x - \tilde{r}_{x_\alpha} \leq \frac{1}{20}|x - x_\alpha| \leq \frac{k}{20}(\tilde{r}_x + \tilde{r}_{x_\alpha}).
    \end{equation}
    Rearranging the inequality, we obtain
    \begin{gather}
    \left(1 - \frac{k}{20}\right) \tilde{r}_x \leq \left(1 + \frac{k}{20}\right) \tilde{r}_{x_\alpha},
    \end{gather}
    which implies the existence of a constant \( w = \frac{1 + \frac{k}{20}}{1 - \frac{k}{20}} > 1 \) such that
    \begin{gather}
    \tilde{r}_x \leq w \tilde{r}_{x_\alpha}.
    \end{gather}
    Reversing the roles of \( x \) and \( x_\alpha \) yields the reverse inequality, and thus
    \begin{gather}
    \frac{1}{w} \tilde{r}_{x_\alpha} \leq \tilde{r}_x \leq w \tilde{r}_{x_\alpha},
    \end{gather}
    establishing the desired two-sided comparability.
\end{proof}

We can now prove the existence and some first result for our partition of unity. This is a standard argument, but we report here for the reader's convenience.

\begin{lemma}\label{partition}
    There exists a covering $B_2\subset \bigcup_\alpha B_{r_\alpha}(x_\alpha)$ and $\tilde{r}_\alpha = \frac{r_{x_\alpha} \vee r}{100}$ and smooth nonnegative functions $\phi_\alpha$ such that 
    \begin{enumerate}
      \item $\{B_{\frac{1}{4}\tilde{r}_\alpha}(x_\alpha)\}$ are pairwise disjoint; 
      \item \label{it_2} For each $y\in B_2$ we have $\#\{x_\alpha \vert \;\; y\in B_{4\tilde{r}\alpha}(x_\alpha)\}\leq C(n)$;
      \item $\sum\phi_\alpha =1$ on $B_2$ and $supp(\phi_\alpha) \subset B_{4\tilde{r}_\alpha}$;
      \item \label{it_4} for all $k$: $\tilde{r}_\alpha^{k} \abs{\partial^{(k)}\phi_\alpha} \leq C(n,k)$. 
    \end{enumerate}
  \end{lemma}

\begin{proof}
    \par{
      We take $\cur{x_\alpha} \subset B_2$ to be a maximal subset so that $\{B_{\frac{\tilde{r}_{x_\alpha}}{4}}(x_\alpha)\}$ are disjoint. We show that $B_2\subset \bigcup_\alpha B_{\tilde{r}_{x_\alpha}}(x_\alpha)$. Fix $y\in B_2$: by maximality, we can find $\alpha$ such that $B_{\frac{\tilde{r}_y}{4}}(y)\cap B_{\frac{\tilde{r_\alpha}}{4}}(x_\alpha) \neq \emptyset$. Using the third point of Lemma \ref{mybound} with $k = \frac{1}{4}$ we can conclude that $\tilde{r}_y \leq 2 \tilde{r}_{x_\alpha}$ and since $|x_\alpha-y| \leq \frac{\tilde{r}_y + \tilde{r}_{x_\alpha}}{4}$ it follows $|x_\alpha-y| \leq \frac{3\tilde{r}_{x_\alpha}}{4}$, proving that $\forall y \in B_2$ there exists ${x_\alpha}$ such that $y \in B_{\tilde{r}_{x_\alpha}}(x_\alpha)$.
    }

\par{
  Let $y\in B_2$ and let $\{x_\beta\}_{1}^N$ the set of centers such that $y\in B_{\tilde{r}_{x_\beta}}(x_\beta)$. This implies that $\frac{1}{2}r_y\leq r_\beta\leq 2r_y$. In particular, the set $\{B_{\tilde{r_\beta}/10}(x_\beta)\}$ is a set of subsets of $B_{8\tilde{r}_y}$. Moreover, these are all disjoint. Thus, by computing the volumes, we obtain $N\leq C(n)$, as claimed. 
}
\par{
  We now build the family of functions $\phi_\alpha$ using a standard partition of unity construction. Let $\psi:B_4\to \R$ a fixed smooth, compactly supported and nonnegative function with $\psi\equiv 1$ in $B_1$. We then define
  \begin{equation}
    \psi_\alpha(x) = \psi\left(\tilde{r}_{x_\alpha}^{-1}(x-x_\alpha)\right)
  \end{equation}
  and
  \begin{equation}
    \phi_\alpha(x)= \frac{\psi_\alpha(x)}{\sum_{\beta}\phi_\beta(x)}. 
  \end{equation}
  This is well defined since $\sum_\beta \phi_\beta(x)\geq 1$ for all $x\in B_2$. Then, automatically we get $\sum_\alpha\phi_\alpha\equiv 1$. Moreover, since $supp(\psi)\subset B_1$, we have $\text{supp}(\psi_\alpha)\subset B_{4\tilde{r}_\alpha}(x_\alpha)$. Lastly, 
  \begin{equation}
    \partial_{y^j}\phi_\alpha = \frac{\partial_{y^j}\psi_\alpha\ton{\sum_\beta \psi_\beta} -\ton{\sum_\beta\partial_{y^j}\psi_\beta}\psi_\alpha}{\ton{\sum_\beta \phi'_\beta(x)}^2}.
  \end{equation}
  Now, we have
  \begin{equation}
    \partial_{y^j}\psi_\alpha = \tilde{r}_{x_\alpha}^{-1}(\partial_{y^j}\psi)(\tilde{r}_{x_\alpha}^{-1}(x-x_\alpha)). 
  \end{equation}
  Thus, 
  \begin{equation}
    \abs{\partial_{y^j}\psi_{\alpha}}\leq \tilde{r}_{x_\alpha}^{-1}C(n).
  \end{equation}
  Moreover, take $x\in B_2$: thus, there is $x_\alpha$ such that $x\in B_{\tilde{r}_\alpha}(x_\alpha)$, and by $\psi\equiv 1$ on $B_1$ we have that $\psi_\alpha(x) = 1$. Moreover, for all $x\in B_2$: 
  \begin{equation}
    \sum_{\beta}\psi_\beta(x) = \sum_{\beta\in I_x}\psi_\beta(x)
  \end{equation}
  with $I_x=\cur{\left. \beta\right\rvert\;\; x\in B_{\tilde{r}_\beta}(x_\beta)}$. Then, 
  \begin{equation}
    \sum_{\beta\in I_x}\psi_\beta(x)\leq C(n)\#\{I_x\}\leq C(n)
  \end{equation}
  by $(\ref{it_2})$. Thus, putting together the two previous estimates we get
  \begin{equation}
    \abs{\partial_{y^j}\phi_\alpha}\leq \tilde{r}_{x_\alpha}^{-1}C(n). 
  \end{equation}
}
\par{
  which proves (\ref{it_4}) for $k=1$. A similar argument can be used to show the estimate for generic $k$. 
}
  \end{proof}

\subsection{Subspace selection lemma}
We can move to the next point where we prove and define the map associating an approximating plane of $S$ to every point of $B_2(0)$. The idea is to use the partition of unity to smoothly "average" the approximating subspaces, thereby assigning to each point a plane that varies smoothly with the point. Morally, one might attempt to define
\begin{gather}
    L_y = \sum_\alpha \phi_\alpha(y) L_\alpha,
\end{gather}
where \( \{\phi_\alpha\} \) is a partition of unity. However, since the Grassmannian manifold of subspaces is not linear, this expression is not generally well-defined. A common workaround, employed for instance in \cite{simon_squash}, is to average the orthogonal projections onto the subspaces instead of the subspaces themselves.

In this work, however, we follow a different approach, inspired by \cite{Naber_Park}. We present here the main technical lemma, which, roughly speaking, tells us that we can assign a $k-$dimensional subspace $L_y$ to each point $y\in B_2$, with quantitative control on its variation wrt $y$..

\begin{lemma}[{Subspace Selection Lemma}, {\cite[Theorem 4.1]{Naber_Park}}]\label{lemma_subsel}
Let \( S \subseteq B_2 \subseteq \mathbb{R}^n \) satisfy condition \eqref{semireif} for $\delta > 0$ with \( 0 < r < 1 \) fixed, and let \(\overline{r_y} = 10^4 \tilde{r}_y = 10^2 r_y \vee r\). Then for each \( y \in B_2 \), there exists a \( k \)-dimensional affine subspace \( L_y \), satisfying 
\begin{equation}
    \overline{r_y} |\nabla \hat{\pi}_y| + \overline{r_y}^2 |\nabla^2 \hat{\pi}_y| \leq C(n)\delta. \quad \forall y \in B_2\, ,
\end{equation}
where $\hat{\pi}_y = \hat{\pi}_{L_y}$ is the projection onto $L_y$.
\end{lemma}

\begin{proof}      
We start by considering a collection of points $\cur{x_\alpha}\subset B_2$ as in Lemma \ref{partition}, and we define
\begin{equation}
    L_{\alpha} \equiv L_{x_{\alpha}, 10^4 \tilde{r}_{\alpha}}\, .
\end{equation}

Notice that by point \textit{(2)} of Theorem \ref{rthm}, $\B {r_\alpha}{x_\alpha}$ has non-empty intersection with $S$.
In order to define $L_y$, we will define the linear subspace $\hat L_y$ relative to the projection map $\hat{\pi}_y$, and a point $\ell_y \in L_y$. Specifically we set:
\begin{equation}\label{def_elly}
    \ell_y \equiv \sum_{\alpha} \phi_{\alpha}(y) \pi_{\alpha}[y],
\end{equation}
where $\pi_{\alpha} =\pi_{x_\alpha, 10^4 \tilde{r}_{\alpha}}$. 

The definition of $\hat{\pi}_y$ is more involved. The rough idea is the following: consider
\begin{equation}
    M_y \equiv \sum_{\alpha} \phi_{\alpha}(y) \hat{\pi}_{\alpha}\, ,
\end{equation}
where here we define $\hat{\pi}_{\alpha}$ as the projection on the linear subspace relative to $L_\alpha$ denoted by $\hat{L}_\alpha$. Even though $M_y$ is not a projection map, it is a symmetric operator in $\R^n$. For a fixed $y$, all the $\hat{\pi}_{\alpha}$ in the definition of $M_y$ are close to each other, and this implies that $M_y$ is close to a projection on the ``average of $L_\alpha$''. The more accurate way to phrase this argument is in terms of an eigenvalue gap for $M_y$, which will allow us to associate to $M_y$ a linear subspace $\hat L_y$ defined in terms of this eigenvalue gap.

To be more precise, we now prove some useful properties that will help us transitioning to the final projection matrix $\hat{\pi}_y$. Let us consider the set $\{x_{\beta}\}$ of points satisfying $y \in B_{8 \tilde{r}_\beta}(x_{\beta})$, and suppose that this set contains more than a single point in order to avoid trivial cases. Since clearly $\cur{y}\subset B_{8 \tilde{r}_{\beta_1}}(x_{\beta_1}) \cap B_{8 \tilde{r}_{\beta_2}}(x_{\beta_2}) \neq \emptyset$, we can apply Lemma \ref{mybound} with $k = 8$ to prove that the radia are comparable, in particular $\frac{1}{10}\tilde{r}_{\beta_1} \leq \tilde{r}_{\beta_2} \leq 100 \tilde{r}_{\beta_1}$.
Now we use Lemma \ref{lemma_planes} with $s_1 = 10^4 \tilde{r}_{\beta_1} = 100 r \vee r_{x_{\beta_1}} \geq 10 r \vee r_{x_{\beta_1}}$ and $ s_2 = 10^4 \tilde{r}_{\beta_2} = 100 r \vee r_{x_{\beta_2}} \geq 10 r \vee r_{x_{\beta_2}}$ to obtain
\begin{gather}
    \|\hat{\pi}_{\beta_1} - \hat{\pi}_{\beta_2}\|\equiv \|\hat{\pi}_{x_{\beta_1}, s_1} - \hat{\pi}_{x_{\beta_2}, s_2}\| < C(n)\delta\, ,
\end{gather}
where we have set for convenience of notation $\hat{\pi}_{\beta_i} \equiv \hat{\pi}_{x_{\beta_i}, s_i}$.

Given $\beta_i$ such that $y \in B_{8 \tilde{r}_{\beta_i}}(x_{\beta_i})$:
\begin{equation}\label{eq_My-pi}
    \left\| M_y - \hat{\pi}_{\beta_i}\right\| = \left\|  \sum_{\alpha} \phi_{\alpha}(y) \hat{\pi}_{\alpha} - \hat{\pi}_{\beta_i}\right\| = \left\|  \sum_{j} \phi_{\beta_j}(y) \hat{\pi}_{\beta_j} - \hat{\pi}_{\beta_i}\right\| 
\end{equation}
and since $\sum_{j} \phi_{\beta_j}(y) = 1$
\begin{equation}
    \left\|  \sum_{j} \phi_{\beta_j}(y) \hat{\pi}_{\beta_j} - \hat{\pi}_{\beta_i}\right\|  = \left\|  \sum_{j} \phi_{\beta_j}(y) \hat{\pi}_{\beta_j} - \left(\sum_{j} \phi_{\beta_j}(y)\right)\hat{\pi}_{\beta_i}\right\| = \sum_{j \neq i} \phi_{\beta_j}(y) \|\hat{\pi}_{\beta_j}- \hat{\pi}_{\beta_i}\| 
\end{equation}
leading to inequality 
\begin{equation}\label{stimaM_y}
    \left\| M_y - \hat{\pi}_{\beta_i}\right\| = \sum_{j \neq i} \phi_{\beta_j}(y) \|\hat{\pi}_{\beta_j}- \hat{\pi}_{\beta_i}\| \leq \sum_{j \neq i} \phi_{\beta_j}(y) C(n) \delta  \leq C(n) \delta\, .
\end{equation}

Moreover, an important feature of $M_y$ is that it presents an eigenvalue gap, where the first $k$ eigenvalues $\{\lambda_i\}_{i\leq k} \in [1-C(n)\delta, 1]$ while the last $n-k$ eigenvalues $\{|\lambda_i|\}_{k < i\leq n} \in [0,C(n)\delta]$, this follows easily from \eqref{eq_My-pi}.

Since $M_y$ is symmetric, we can diagonalize it and consider its eigenvectors $\{v^1(y), v^2(y), ...\}$, ordered in such a way that the corresponding eigenvalue $\{\lambda_1(y), \lambda_2(y), ...\}$ are non-increasing. To address the fact that $M_y$ is not an orthogonal projection matrix we consider the actual projection on the following subspace:
\begin{equation}\label{hatpi}
    \hat{L}_y \equiv \text{span}\{v^1(y), v^2(y), ..., v^k(y)\}.
\end{equation}
called $\hat{\pi}_y$.

We can now leverage property \eqref{stimaM_y} to get some bounds on the norm of new projection matrix using $M_y$, in fact:
\begin{equation}
    \|M_y - \hat{\pi}_y\|
=
\left\| \sum_{j=1}^k (1-\lambda_j) \, v_j \, v_j^T + \sum_{j=k+1}^n \lambda_j \, v_j \, v_j^T\right\| \leq \sum_{j=1}^k  |1-\lambda_j| + \sum_{j=k+1}^n  |\lambda_j|\leq C(n)\delta
\end{equation}
and thus, applying triangle inequality along with \eqref{stimaM_y} one can also conclude that
\begin{equation}
    \|\hat{\pi}_y - \hat{\pi}_{\beta_i}\| \leq C(n) \delta
\end{equation}
for every $y\in B_{8\tilde{r}_{\beta_i}}(x_{\beta_i})$.
Consequently if we consider $y\in B_{8\tilde{r}_{\beta}}(x_{\beta}) \cap  B_3$ it also holds
\begin{equation}
    \left|M_y[y]-\hat{\pi}_{\beta}[y]\right|,\left|\hat{\pi}_y[y]-\hat{\pi}_{\beta}[y]\right| \leq C(n)\delta \tilde{r}_y.
\end{equation}

\textbf{IFT: Implicit Function Theorem.}  
We now show that the map $\hat \pi_y$ arises as the zero of a suitably defined function, and we apply the implicit function theorem to derive the desired properties.

Let $V$ denote the space of linear maps $v : L_\beta \to L_\beta^\perp$, and define the subset
\begin{gather}\label{condv}
V_s := \{ v \in V : \|v\| \leq 0.1 \},
\end{gather}
where the norm bound ensures proximity to the correct stationary point.

We define a smooth function
\begin{gather}
F : B_{4\tilde{r}_\beta}(x_\beta) \times V_s \to V
\end{gather}
by the condition that for all $w\in L_\beta$:
\begin{gather}\label{F}
\langle F(y, v), w \rangle := \partial_w \operatorname{tr}_{\hat{L}_v}(M_y) = \partial_w \sum_i \langle e_i, M_y[e_i] \rangle,
\end{gather}
here, $\hat{L}_v$ denotes the affine subspace determined by the graph of $v$, $\partial_w$ denotes the directional derivative in the direction $w \in L_\beta$, and $\{e_i\}_{i=1}^k$ is an orthonormal basis of the graph $L_v$.

In other words, the function $F$ is the gradient of $\operatorname{tr}_{\hat{L}_v}(M_y)$ with respect to variations of the plane $\hat{L}_v$.

We will establish the following estimates:
\begin{gather}\label{eq_claim_F}
\tilde r_y \| \partial_{y^i} F(y,v) \| + \tilde r_y^2 \| \partial_{y^i} \partial_{y^j} F(y,v) \| \leq C(n)\delta, \quad
| \langle \partial_v F(y, \hat\pi_y), w \rangle + \langle v, w \rangle | \leq C(n)\delta
\end{gather}
where $\partial_v$ denotes the Jacobian of the function $F$ with respect to the linear application $v$, while $\partial_{y^i}$ and $\partial_{y^j}$ are the partial derivatives relative to the components of the vector $y$.
These bounds allow us to apply the IFT, yielding the desired control on $\hat\pi_y$.

In particular, given any $w\in L_\beta$, $w\neq 0$, we define for convenience
\begin{gather}
    e_w= \frac{(w, v(w))}{|(w, v(w))|}
\end{gather}
Some easy calculations, which for convenience are postponed to section \ref{sec_IFT} show that
\begin{equation}
 \partial_w tr_{\hat{L}_v}(M_y) =  2 \langle e_w, M_y[u]\rangle = 0,
\end{equation}
showing that the subspace associated to $\hat\pi_y$ is a stationary point. Using the same estimate, we obtain
\begin{equation}
    \partial_{y^i} \partial_w \operatorname{tr}_{L_v}(M_y) = 2 \langle e_w, \partial_{y^i} M_y[u] \rangle\, .
\end{equation}
Since this holds for every fixed vector \( w \in L_\beta \), it follows that
\begin{equation}
    \| \partial_{y^i} F(y,v) \| \leq C(n) \| \partial_{y^i} M_y \|\, .
\end{equation}
Observing that
\begin{equation}
    \| \partial_{y^i} M_y \| = \left\| \sum_\alpha \partial_{y^i} \phi_\alpha(y) \hat{\pi}_\alpha \right\| = \left\| \sum_\alpha \partial_{y^i} \phi_\alpha(y)(\hat{\pi}_\alpha - \hat{\pi}_\beta) \right\| \leq C(n) \delta \tilde{r}_y^{-1}\, ,
\end{equation}
we obtain bound
\begin{equation}
    \| \partial_{y^i} F(y,v) \| \leq C(n) \delta \tilde{r}_y^{-1}\,
\end{equation}
on the first derivative of \( F \). Similarly, one can derive the second derivative bounds:
\begin{gather}
    \| \partial_{y^i} \partial_{y^j} F(y,v) \| \leq C(n) \delta \tilde{r}_y^{-2}.
\end{gather}

Now we move to the last estimate in \eqref{eq_claim_F}. By direct computation, we find
\begin{equation}\label{he}
    \frac{1}{2} \partial_w^2 \operatorname{tr}_{\hat{L}_v}(M_y) = -\langle e_w, M_y[e_w] \rangle + \langle u, M_y[u] \rangle.
\end{equation}
Using the spectral gap of \( M_y \), we note that if \( e_1 \in \operatorname{span}\{v_j\}_{j=1}^k \), then \( \langle e_1, M_y[e_1] \rangle \geq \lambda_k \), while for \( u \in \operatorname{span}\{v_j\}_{j=k+1}^n \), we have \( \langle u, M_y[u] \rangle \leq \lambda_{k+1} \). This yields the estimate
\begin{gather}
    \frac{1}{2} \partial_w^2 \operatorname{tr}_{L_y}(M_y) \leq \lambda_{k+1} - \lambda_k \leq - (1 - C(n) \delta)\, .
\end{gather}

Since equation~\eqref{he} holds for every \( w \in L_\beta \) and \( u \in \{u_j\}_{j=k+1}^n \), and noting that \( \pi_y(w) = e_1 \in \operatorname{span}\{v_j\}_{j=1}^k \), we conclude that the Jacobian of \( F \) is close to the opposite of the identity in the sense that
\begin{equation}
    |\langle \partial_v F(y, \hat\pi_y), w \rangle + \langle v, w \rangle | \leq C(n)\delta.
\end{equation}

With these inequalities we can apply implicit function theorem showing that the function $\hat{\pi}_y$ satisfies
\begin{equation}
      \tilde{r}_y\|\nabla \hat{\pi}_y\|, \tilde{r}_y^{2}\|\nabla^2 \hat{\pi}_y\| \leq C(n)\delta.
\end{equation}
Given the definition of the radius $\overline{r}_y$ the thesis follows with different constants as there is a constant fixed factor for the ratio between $\tilde{r}_y$ and $\overline{r}_y$.
\end{proof}

This concludes the main result of the Subspace Selection Lemma. We now define an affine map \(\pi_y\) centered at the point \(\ell_y\), using the projection \(\hat{\pi}_y\) constructed in Theorem~\ref{lemma_subsel}. Specifically, we set $\forall z$:
\begin{gather}\label{eq_deph_my}
\pi_y[z] := \hat{\pi}_y[z - \ell_y] + \ell_y, \quad \text{and define} \quad m_y := \pi_y[y].
\end{gather}
This map projects the point \(y\) onto the approximate tangent plane \(L_y\) centered at \(\ell_y\), and we use it to describe the local geometric behavior of the set. The following proposition provides bounds on the first and second derivatives of \(m_y\) in terms of \(\hat{\pi}_y\):

\begin{proposition}\label{lemma_subsel_m}
Given the above defined quantities $\pi_y$,$m_y$,$\delta$ and $\overline{r}_y$, the following holds:
\begin{equation}
     \|\nabla m_y - \hat{\pi}_y\| +\overline{r_y} \|\nabla^2 m_y\| \leq C(n)\delta. \quad \forall y \in B_2
\end{equation}
\end{proposition}

\begin{proof}
We will assume the following technical estimates (Claims 1 and 2) and defer their proofs to Appendix~\ref{Claims}, as they follow from direct computations using the constructions in Theorem~\ref{lemma_subsel}:

\begin{itemize}
    \item \textbf{Claim 1}: 
    \begin{gather}\label{eq_claim1}
    \left| \partial_{y^i}(\ell_y - \hat{\pi}_y[y]) \right| \leq C(n)\delta, \quad 
    \left| \partial_{y^i} \partial_{y^j} \ell_y \right| \leq C(n)\delta \tilde{r}_y^{-1}
    \end{gather}
    
    \item \textbf{Claim 2}: 
    \begin{gather}\label{eq_claim2}
    \left| \partial_{y^i}(M_y[y] - \hat{\pi}_\beta[y]) \right| \leq C(n)\delta, \quad 
    \left| \partial_{y^i} \partial_{y^j} M_y[y] \right| \leq C(n)\delta \tilde{r}_y^{-1}
    \end{gather}
\end{itemize}

Here, the maps \( \ell_y, M_y, \hat{\pi}_\beta, \hat{\pi}_y \) are defined in Theorem~\ref{lemma_subsel}. The estimates follow from straightforward but technical computations. 

\medskip

\textbf{Proof of the first inequality: } here we prove $\|\nabla m_y - \hat{\pi}_y\|\leq C(n)\delta$. To begin with, we estimate the first derivative of the map \( m_y = \hat{\pi}_y[y - \ell_y] + \ell_y \). Differentiating, we find:
\begin{gather}
\left| \partial_{y^i}(m_y - \hat{\pi}_y[y]) \right| 
= \left| \partial_{y^i}\left( \hat{\pi}_y[y - \ell_y] + \ell_y - \hat{\pi}_y[y] \right) \right|
\leq \left| \partial_{y^i} \ton{\hat{\pi}_y[y - \ell_y] }\right| + \left| \partial_{y^i} (\ell_y - \hat{\pi}_y[y]) \right|.
\end{gather}

The second term is controlled directly using Claim 1, i.e. \eqref{eq_claim1}. To handle the first term, we expand:
\begin{equation} \label{sin}
\left| \partial_{y^i} \hat{\pi}_y[y - \ell_y] \right|
= \left| (\partial_{y^i} \hat{\pi}_y)[y - \ell_y] + \hat{\pi}_y\left[e_i - \partial_{y^i} \ell_y\right] \right|.
\end{equation}

We now estimate the two terms in \eqref{sin} separately.

\medskip

\emph{Estimate for the first term:}  
Since \( y - \ell_y = \sum \phi_\alpha \pi_\alpha^{\perp}[y] \), and \( |\pi_\alpha^{\perp}[y]| \leq C(n)\tilde{r}_y \) for all $\alpha$, we have:
\begin{gather}
|y - \ell_y| \leq C(n)\tilde{r}_y.
\end{gather}
Using the bound on \( \partial_{y^i} \hat{\pi}_y \) from Theorem~\ref{lemma_subsel}, we obtain:
\begin{gather}
|(\partial_{y^i} \hat{\pi}_y)[y - \ell_y]| \leq C(n)\delta.
\end{gather}

\medskip

\emph{Estimate for the second term:}  
We now estimate \( \hat{\pi}_y[e_i - \partial_{y^i} \ell_y] \) via the decomposition:
\begin{equation}
\begin{split}
\left| \hat{\pi}_y\left[e_i - \partial_{y^i} \ell_y \right] \right|
&= \left| \hat{\pi}_y\left[e_i - \sum_\alpha \partial_{y^i} \phi_{\alpha}(y) \pi_\alpha[y] - \sum_\alpha \phi_{\alpha}(y) \hat{\pi}_\alpha[e_i] \right] \right| \\
&= \left| \hat{\pi}_y\left[e_i - \sum_\alpha \partial_{y^i} \phi_{\alpha}(y) \pi_\alpha[y] - M_y[e_i] \right] \right| \\
&\leq \left| \hat{\pi}_y[e_i - M_y[e_i]] \right| + \left| \hat{\pi}_y\left[ \sum_\alpha \partial_{y^i} \phi_{\alpha}(y) \pi_\alpha[y] \right] \right|.
\end{split}
\end{equation}

The second term is small because it is a projection of a linear combination of controlled terms:
\begin{gather}
\left| \hat{\pi}_y\left[ \sum_\alpha \partial_{y^i} \phi_{\alpha}(y) \pi_\alpha[y] \right] \right| \leq C(n)\delta.
\end{gather}

\emph{To estimate \( \hat{\pi}_y[e_i - M_y[e_i]] \)} it's enough to notice that $M_y$ is "close" to $\hat \pi_y$ and therefore $\hat \pi_y \circ M_y \approx \hat \pi_y$ in a controlled manner. To show this, we recall that \( M_y \) is diagonalizable with eigenbasis \( v_1, \dots, v_n \) and eigenvalues \( \lambda_1, \dots, \lambda_n \), satisfying \( \lambda_j \geq 1 - C(n)\delta \) for \( j \leq k \). Writing
\begin{gather}
e_i = \sum_{j=1}^n \alpha_j v_j, \qquad M_y[e_i] = \sum_{j=1}^n \alpha_j \lambda_j v_j,
\end{gather}
we have
\begin{gather}
\hat{\pi}_y[e_i] = \sum_{j=1}^k \alpha_j v_j, \qquad \hat{\pi}_y[M_y[e_i]] = \sum_{i=1}^k \alpha_j \lambda_j v_j,
\end{gather}
so
\begin{gather}
\hat{\pi}_y[e_i - M_y[e_i]] = \sum_{j=1}^k \alpha_j(1 - \lambda_j) v_j.
\end{gather}
Each \( |1 - \lambda_j| \leq C(n)\delta \), hence
\begin{gather}
\left| \hat{\pi}_y[e_i - M_y[e_i]] \right| \leq C(n)\delta.
\end{gather}

\medskip

Combining all bounds, we conclude:
\begin{gather}
\left| \partial_{y^i}(m_y - \hat{\pi}_y[y]) \right| \leq C(n)\delta,
\end{gather}
as desired.

\medskip

\textbf{Second inequality: }here we prove $\overline{r_y} \|\nabla^2 m_y\| \leq C(n)\delta$. We now estimate the second derivatives of \( m_y \). Differentiating twice, we find:
\begin{equation}
\begin{split}
\left| \partial_{y^i} \partial_{y^j} m_y \right| 
&= \left| \partial_{y^i} \partial_{y^j} \left( \hat{\pi}_y[y - \ell_y] + \ell_y \right) \right| \\
&\leq \left| \partial_{y^i} \partial_{y^j}\ton{ \hat{\pi}_y[y - \ell_y]} \right| + \left| \partial_{y^i} \partial_{y^j} \ell_y \right|.
\end{split}
\end{equation}

The second term is controlled by \eqref{eq_claim1}. To estimate the first, we use the identity:
\begin{equation}
\begin{split}
\partial_{y^i} \partial_{y^j} \left( \hat{\pi}_y[y - \ell_y] \right)
&= (\partial_{y^i} \partial_{y^j} \hat{\pi}_y)[y - \ell_y] 
+ (\partial_{y^i} \hat{\pi}_y)\left[e_j - \partial_{y^j} \ell_y\right] \\
&\quad + (\partial_{y^j} \hat{\pi}_y)\left[e_i - \partial_{y^i} \ell_y\right] 
- \hat{\pi}_y[\partial_{y^i} \partial_{y^j} \ell_y].
\end{split}
\end{equation}

Each term is controlled using mainly the bounds in Lemma \ref{lemma_subsel} as follows:
\begin{itemize}
    \item The first term uses bounds on \( \partial_{y^i} \partial_{y^j} \hat{\pi}_y \) and the fact that \( |y - \ell_y| \leq C(n)\tilde{r}_y \).
    \item The second and third terms are bounded via \( \partial_{y^i} \hat{\pi}_y \) and \( e_k - \partial_{y^k} \ell_y \in B_2 \), for \( k = i, j \).
    \item The last term is handled using \eqref{eq_claim1}: \( |\partial_{y^i} \partial_{y^j} \ell_y| \leq C(n)\delta \tilde{r}_y^{-1} \).
\end{itemize}

Putting all terms together, we conclude:
\begin{gather}
\left| \partial_{y^i} \partial_{y^j} m_y \right| \leq C(n)\delta \tilde{r}_y^{-1}.
\end{gather}
\end{proof}

As a consequence, the approximating subspaces \( \hat{L}_y \) inherit quantitative bounds derived from this construction.

\begin{corollary}\label{lemma_subsel_planes}
Given $L_y$,$\overline{r}_y$,$\delta$ defined above, the following holds:
\begin{equation}
d_\cH(S \cap B_{\overline r_{y}}(y), L_{y} \cap  B_{\overline  r_{y}}(y)) \leq C(n)\delta \overline r_{y}    
\end{equation}
and
\begin{equation}
    \mathrm{d}_\cH( L_y \cap B_{10 \overline{r}_y}(y),\; L_{y, 10^5 \bar{r}_y} \cap B_{10 \overline{r}_y}(y) ) \leq C(n) \delta \overline{r}_y
\end{equation}
\end{corollary}

\begin{proof}
Recall that the construction of the subspaces \(L_y\) and \(L_\beta\) was given in Lemma~\ref{lemma_subsel}.

We aim to show that \(L_y\) and \(L_\beta\) are close in the Hausdorff sense inside \(B_{10\overline{r}_y}(y)\), which will yield the desired bound.

Recall that from the definition in \ref{def_elly} we have:
\begin{gather}
\ell_y \equiv \sum_{\alpha} \phi_\alpha(y)\, \pi_\alpha[y],
\end{gather}
which lies in \(L_y\), and consider its projection onto \(L_\beta\), namely \(\pi_\beta[\ell_y]\). We estimate their distance:
\begin{align*}
|\ell_y - \pi_\beta[\ell_y]| 
&= \left| \sum_{\beta'} \phi_{\beta'}(y)\, \pi_{\beta'}[y] - \pi_\beta\left( \sum_{\beta'} \phi_{\beta'}(y)\, \pi_{\beta'}[y] \right) \right| \\
&= \left| \sum_{\beta'} \phi_{\beta'}(y)\left( \pi_{\beta'}[y] - \pi_\beta[\pi_{\beta'}[y]] \right) \right| \\
&= \left| \sum_{\beta'} \phi_{\beta'}(y)\left( \pi_{\beta'}[\pi_{\beta'}[y]] - \pi_\beta[\pi_{\beta'}[y]] \right) \right| \\
&\leq \sum_{\beta'} |\phi_{\beta'}(y)| \|\pi_{\beta'} - \pi_\beta\| |\pi_{\beta'}[y]| \\
&\leq C(n)\delta\, \overline{r}_y.
\end{align*}

To estimate the full Hausdorff distance between \(L_y\) and \(L_\beta\), we recall that both are affine \(k\)-planes, and the projections \(\pi_y\) and \(\pi_\beta\) satisfy \(\|\pi_y - \pi_\beta\| \leq C(n)\delta\). For any \(x \in L_y \cap B_{10\overline{r}_y}(y)\), we may write \(x = \ell_y + v\) for some \(v \in L_y\) with \(|v| \leq 10\overline{r}_y\). Then,
\begin{gather}
\operatorname{dist}(x, L_\beta) \leq |\ell_y - \pi_\beta[\ell_y]| + \| \pi_y - \pi_\beta \| |v| \leq C(n)\delta\, \overline{r}_y.
\end{gather}
A symmetric argument shows that any point in \(L_\beta \cap B_{10\overline{r}_y}(y)\) is within \(C(n)\delta\, \overline{r}_y\) of \(L_y\), so we conclude:
\begin{gather}
d_\mathcal{H}(L_y \cap B_{10\overline{r}_y}(y),\; L_\beta \cap B_{10\overline{r}_y}(y)) \leq C(n)\delta\, \overline{r}_y.
\end{gather}

To compare \(L_y\) with the larger-scale best approximating plane \(L_{y,10^5\overline{r}_y}\), we again apply the triangle inequality:
\begin{align*}
d_\mathcal{H}(L_y \cap B_{10\overline{r}_y}(y),\; L_{y,10^5\overline{r}_y} \cap B_{10\overline{r}_y}(y))
&\leq d_\mathcal{H}(L_y \cap B_{10\overline{r}_y}(y),\; S \cap B_{10\overline{r}_y}(y)) \\
&+ d_\mathcal{H}(S \cap B_{10\overline{r}_y}(y),\; L_{y,10^5\overline{r}_y} \cap B_{10\overline{r}_y}(y)).
\end{align*}
The first term is bounded by the estimate above, and the second term is controlled by the Reifenberg condition at scale \(10^5 \overline{r}_y\). Since
\begin{gather}
S \cap B_{10\overline{r}_y}(y) \subset S \cap B_{10^5\overline{r}_y}(y),
\end{gather}
we get
\begin{gather}
d_\mathcal{H}(S \cap B_{10\overline{r}_y}(y),\; L_{y,10^5\overline{r}_y} \cap B_{10\overline{r}_y}(y)) \leq C(n)\delta \cdot 10^4 \overline{r}_y.
\end{gather}
Hence,
\begin{gather}
d_\mathcal{H}(L_y \cap B_{10\overline{r}_y}(y),\; L_{y,10^5\overline{r}_y} \cap B_{10\overline{r}_y}(y)) \leq C(n)\delta\, \overline{r}_y + C(n)\delta \cdot 10^4 \overline{r}_y \leq C(n)\delta \overline{r}_y.
\end{gather}
\end{proof}

\par{
We are now ready to define a function that will serve as a key tool in building the \emph{approximating manifold} \( S_r \). The idea is to quantify how far a point is from its corresponding approximating plane \( L_y \), and use this to produce a regular geometric object that reflects the structure of the set \( S \) at scale \( r \). To this end, we define the function \( \Phi_r \) as follows:
\begin{equation}
    \Phi_r(y) = \frac{1}{2}\abs{y - \pi_y[y]}^2.
\end{equation}
This function measures the squared distance from \( y \) to the plane \( L_y \), and as we will see, it encodes enough regularity to extract geometric information from its structure.

The next lemma, adapted from Theorem 4.12 in \cite{Naber_Park}, summarizes the key properties of \( \Phi_r \). These will be instrumental in proving that the level set \( S_r := \nabla \Phi_r^{-1}(0) \) defines a smooth manifold that locally approximates the set \( S \), and provides control over its geometry in terms of the smallness parameter \( \delta \).

\begin{lemma}[\cite{Naber_Park}, Theorem 4.12]\label{lemma:phi}
    Given $\Phi_r$ as above, for each $y\in B_2$ the following holds:
    \begin{enumerate}[(1)]
        \item For every $x\in S$ and $\ell \in L_x\cap B_r(x)$ there is a unique $z_\ell\in \hat{L}_x+ \ell$ such that $\phi_r(z_\ell) = 0$, moreover, $\abs{z_\ell - \ell}\leq C(n)\delta \overline r_\ell$.
        \item\label{it_2bis} $\abs{\abs{\nabla \Phi_r}^2 - 2\Phi_r}(y) \leq C(n)\delta\Phi_r(y)$;
        \item\label{it_3} $\abs{\nabla^2\Phi_r(y) - \hat{\pi}^\perp (y)}\leq C(n)\delta$.
    \end{enumerate}
\end{lemma}

\begin{remark}
 As a consequence of point \eqref{it_2bis} in this lemma:
 \begin{gather}
  S_r:= \nabla \Phi_r^{-1}(0)=\Phi_r^{-1}(0)\, .
 \end{gather}
The first characterization is what will be used in practice to apply the implicit function theorem in order to obtain estimates on $S_r$. The second characterization is perhaps a bit easier to understand, at least at an intuitive level.
\end{remark}

\begin{proof}
    \par{
    We prove $(1)$ by hand, while the rest of the estimates follow by the previous Lemma. 
    }
    \par{
    Let $x\in S$ and $\ell\in L_x\cap B_r(x)$. Let $\xi_r: \hat{L}_x\to \hat{L}_x$ be a smooth cutoff function with $\xi_r \equiv 1$ on $B_r$ and $\xi_r\equiv 0$ outside of $B_{2r}$. 
    }
    \par{
    Now, take $y\in B_r(x)$. Then, by the Subspace Selection Lemma, for some $t\in [0,1]$
    \begin{equation}\label{e: Lipschitz estimate on planes}
        \begin{split}
            \norm{\hat{\pi}_y - \hat{\pi}_x}&\leq \norm{(\nabla\hat{\pi})_{tx + (1-t)y}}\abs{x-y}\leq C(n)\delta\frac{1}{\overline{r}_{tx+(1-t)y}}\abs{x-y}\\
            &\leq C(n)\delta \frac{r}{\overline{r}_{tx +(1-t)y}}\leq C(n)\delta.
        \end{split}
    \end{equation}
    This means that for $\delta \leq \delta_0(n)$, $L_y$ and $\hat{L}_x^\perp + \ell$ are transverse and therefore for every $y\in B_r(x)$ $L_y$ and $\hat{L}_x^\perp+\ell$ intersect at exactly one point $w_y$.
    Moreover, thanks to the Subspace Selection Lemma $L_y$ varies smoothly wrt $y$, and $w_y$ varies smoothly wrt $y$ as well. In fact, note that $w_y$ is the unique solution to a system of linear equations.
    In particular, the condition $w_y\in \hat{L}_x+\ell$ is equivalent to saying that $w_y$ solves a system of $k$ equations in $n$ variables.
    Moreover, the condition $w_y\in L_y$ implies $w_y = \ell_y + \tilde{w}_y$, with $\tilde{w_y}\in \hat{L}_y$. Then, $w_y$ can also be rewritten as the solution of an affine system of $n-k$ equations in $n$ variables, whose coefficients depend smoothly on $y$ as $L_y$ does.
    }
    \par{
    Then, we can rewrite the two systems as one $n\times n$ system 
    \begin{equation}
        A(y) w_y = b_y\, , 
    \end{equation}
    with $A(y)$ a $n\times n$ matrix and $b_y\in \dR^n$. Moreover, note that as $L_y$ and $\hat{L}_x + \ell$ are transversal, $A(y)$ is invertible (this comes from the fact that $A(y)$ is the differential of the matrix of the projection on $\hat{L}_y^\perp$ and the projection on $\hat{L}_x$), which then implies that $w(y) = A(y)^{-1}b_y$. As the inversion and the matrix product are both smooth, $w_y$ is a smooth function of $y$. 
    }
    \par{
    Then, consider the map $\zeta : \hat{L}_x \to \hat{L}_x^\perp$ defined as 
    \begin{equation}
        \zeta(y) = y-\xi_{\overline{r}_x}(y)\ton{w_{y + \ell} - \ell}. 
    \end{equation}
    Note that as $\ell \in B_{\overline{r}_x}(x)$ and as $\zeta(y) =y$ if $y\notin B_{2\overline{r}_x}$, and if $y\in B_{2\overline{r}_x}$ then $y+\ell \in B_{3\overline{r}_x}(x)$ and we can apply the previous observations (still, $3\overline{r}_x\leq C(n)\overline{r}_z$ for any $z$ by the Lipschitz property of the radius and the fact that $\overline{r}_z\geq 10^4 r$ for any $z$) to ensure that $w_{y+\ell}$ exists and is well defined. Moreover, note that the function is smooth as the maps $\xi_{\overline{r}_x}, y\mapsto w_{y+\ell}$ are smooth. Moreover, the map takes indeed value in $\hat{L}_x^\perp$ as $y\in \hat{L}_{x}^\perp,$ $w_{y+\ell}\in \hat{L}_x + \ell$. 
    }
    \par{
    Then, $\zeta$ is smooth and $\zeta(y) = y$ outside of $B_{2\overline r_x}$. Moreover, if $y\in B_{2\overline r_x}$
    \begin{equation}
        \begin{split}
            \abs{\pi(y) -y} & = \abs{\xi_{\overline r_x}(y) (w_{y+\ell} - \ell)}\leq C(n) \abs{w_{y+\ell} - \ell}. 
        \end{split}
    \end{equation}
    Now, note that for any $z$
    \begin{equation}
        \pi_{x}(z) = \ton{z+\hat{L}_{x}^\perp} \cap L_{x}. 
    \end{equation}
    Then, as by definition $w_{y+\ell} = \ton{\hat{L}_{x}^\perp +\ell} \cap L_{y+\ell}$ we have $\ell = \pi_x(w_{y+\ell})$: $\pi_x(w_{y+\ell}) = (w_{y+\ell} + \hat{L}_x^\perp)\cap L_x = \ton{\ell + \hat{L}_x^\perp}\cap L_x = \ell$. Thus, 
    \begin{equation}\label{add:dis}
        \abs{w_{y+\ell} - \ell} = \abs{w_{y+\ell}-\pi_{x}(w_{y+\ell})} \leq C(n)\delta \overline{r}_x
    \end{equation}
    as $\ell\in L_x$, $y+\ell \in B_{3\overline r_x}(x)$. Here we use Corollary \ref{lemma_subsel_planes} and the fact that $r\leq C(n)\overline{r}_\ell$: in fact, as $\overline{r}$ is a Lipschitz function and $y+\ell\in B_{3\overline r_x}(x)$ to conclude that $\overline{r}_x$ and $\overline{r}_{y+\ell}$ are close, and then the first part of Corollary \ref{lemma_subsel_planes}. 
    }
    \par{
    We now estimate the differential: again, outside of $B_{2\overline r_x}$ we have $\zeta = Id$ and therefore $\nabla\zeta = Id$. In $B_{2\overline r_x}$, we can write
    \begin{equation}
        \zeta(y) = y - \xi_{\overline r_x}(y) \ton{w_{y+\ell} -\ell}
    \end{equation}
    and therefore
    \begin{equation}
        \partial_{y^i}(\zeta(y))_j = \delta_{i,j} - \partial_{y^i} \xi_{\overline r_x}(y) (w_{y+\ell} - \ell)_j - \xi_{\overline r_x}(y) \partial_{y^i} \ton{w_{y+\ell} - \ell}_j
    \end{equation}
    and from here we can conclude using the fact that $\abs{\partial_{y^i}\xi_{\overline r_x}}\leq \frac{C(n)}{\overline r_x}$ and the estimates from the Subspace Selection Lemma: in fact, 
    \begin{equation}
        \begin{split}
        \abs{\partial_{y^i} (\zeta(y))_j -\delta_{i,j}} &= \abs{\partial_{y^i}\xi_{\overline r_x}(y) \ton{w_{y+\ell} - \ell}_j + \xi_{\overline r_x}(y)\partial_{y^i}\ton{w_{y+\ell} -\ell}_j}\\
        & \leq \frac{C(n)}{\overline r_x}\abs{w_{y+\ell} - \ell} + C(n)\abs{\partial_{y^i}(w_{y+\ell})_j}
        \end{split}
    \end{equation}
    }
    \par{
    We are left with estimating $\abs{\partial_{y^i}(w_{y+\ell})_j}$. We first check how $L_{y+\ell}$ changes when $y$ changes: as we have already estimated how $\hat{L}_{y+\ell}$ changes in the first part of the proof, we now check how $\ell_{y+\ell}$ changes. We have (fixing $\beta$ such that $x\in B_{4\tilde{r}_\beta}(x_\beta)$)
    \begin{equation}
        \begin{split}
            \partial_{y^i} \ell_{y+\ell} &= \sum \partial_{y^i} \phi_{\alpha}(y+\ell) \pi_\alpha[y+\ell] + \sum  \phi_{\alpha}(y+\ell) \partial_{y^i} \pi_\alpha[y+\ell] \\
            & = \sum \partial_{y^i}\phi_\alpha (y+\ell)\ton{\pi_\alpha[y+\ell] - \pi_\beta[y+\ell]} + \sum \phi_\alpha(y+\ell)\hat{\pi}_\alpha[e_i]. 
        \end{split}
    \end{equation}
    We conclude exactly as in Lemma \ref{lemma_planes} that 
    \begin{equation}
        \abs{\sum \partial_{y^i} \phi_\alpha (y+\ell) \ton{\pi_\alpha \qua{y+\ell} - \pi_\beta[y+\ell]}}\leq C(n)\delta. 
    \end{equation}
    Moreover, note that $\phi_\alpha(y+\ell)\neq 0$ if and only if $y+\ell\in B_{4\tilde r_\alpha}(x_\alpha)$. However, then, as $e_i\in \hat{L}_x^\perp$
    \begin{equation}
        \sum \phi_\alpha (y+\ell) \hat{\pi}_\alpha[e_i] = \sum \phi_\alpha (y+\ell) \ton{\hat{\pi}_\alpha[e_i] -\hat{\pi}_x[e_i]}. 
    \end{equation}
    However, $\hat{\pi}_\alpha$ and $\hat{\pi}_{x_\alpha}$ (where the first is the projection on $\hat{L}_{x_\alpha, 10^4\tilde{r}_\alpha}$ and the second is the projection on the subspace $L_{x_\alpha}$ found through the subspace selection Lemma) are $C(n)\delta -$ close by Corollary \ref{lemma_subsel_planes}. Then, via Lipschitz estimates, we have that $\hat{\pi}_{x_\alpha}$ and $\hat{\pi}_x$ are $C(n)\delta$ close. This proves that 
    \begin{equation}
        \abs{\partial_{y^i}\ell_{y+\ell}} \leq C(n)\delta. 
    \end{equation}
    However, then for $y_1, y_2\in \hat{L}_x^\perp$ (taking into account that $w_{y_j + \ell}\in B_{C(n)\overline{r_x}}(\ell_{y_j+\ell})$)
    \begin{equation}
        \begin{split}
           \abs{w_{y_1 + \ell} - w_{y_2 + \ell}}& \leq \abs{\ell_{y_1 + \ell} - \ell_{y_2 + \ell}} + \overline{r}_x \dist_{Gr}\ton{\hat{L}_{y_1 + \ell}, \hat{L}_{y_2 + \ell}}\\
           & \leq C(n)\delta \abs{y_1 -y_2} + C(n)\overline{r}_x \abs{\hat{\pi}_{y_1 + \ell} - \hat{\pi}_{y_2 + \ell}} \\
           & \leq C(n) \delta \abs{y_1 -y_2} + C(n)\delta \overline{r_x}\frac{\abs{y_1 - y_2}}{\overline{r_x}} \\
           & \leq C(n) \delta \abs{y_1 -y_2}. 
        \end{split}
    \end{equation}
    This proves that $\abs{\partial_{y^i}(w_{y+\ell})}\leq C(n)\delta$. Therefore, 
    \begin{equation}
        \norm{\nabla \zeta - Id}\leq C(n)\delta. 
    \end{equation}
    }
    \par{
    Then, there is $y\sim 0$ such that $\zeta(y) = 0$. However, as $y$ is very close to $0$ we have $\xi_r(y) = 1$ and therefore $\zeta(y) = 0$ becomes
    \begin{equation}
        y = w_{y+\ell} - \ell, 
    \end{equation}
    i.e.
    \begin{equation}
        y+\ell = w_{y+\ell}\in (\hat{L}_x^\perp + \ell)\cap L_{y+\ell}. 
    \end{equation}
    Then, if $z_\ell = y+\ell$, we have $z_\ell \in \hat{L}_x^\perp$ and as $z_\ell \in L_{z_\ell}$ $\pi_{z_\ell}(z_\ell) = z_\ell$, i.e. $\Phi_r(z_\ell) = 0$. This with estimate \ref{add:dis} also shows point $(2)$ of the thesis.
    }
    \vspace{5mm}
    \par{
    We now prove \eqref{it_2bis}-\eqref{it_3}: as said earlier, these estimates mostly follow from the ones in the Subspace Selection Lemma.
    }
    \par{
    We start with \eqref{it_2bis}. For any $i$, we have (using $m_y = \pi_y\qua{y}$ for simplicity)
    \begin{equation}
        \begin{split}
            \partial_{y^i} \Phi_r(y) = \ps{\nabla\ton{y-m_y}[e_i]}{y-m_y}
        \end{split}
    \end{equation}
    where $\nabla\ton{y-m_y}$ is the Jacobian.

    Then,
    \begin{equation}
        \begin{split}
            \partial_{y^i} \Phi_r(y) &= \ps{e_i}{y-m_y} - \ps{\nabla m_y\qua{e_i}}{y-m_y} \, .
        \end{split}
    \end{equation}
    Note that
    \begin{equation}
        y-m_y = y-\ton{\ell_y +\hat{\pi}_y\qua{y-\ell_y}} = y-\ell_y - \hat{\pi}_y\ton{y-\ell_y} = \hat{\pi}_y^\perp\ton{y-\ell_y},
    \end{equation}
    and thus
    \begin{equation}
        \ps{\hat{\pi}_y\qua{e_i}}{y-m_y} = 0
    \end{equation}
    by orthogonality. This implies that
    \begin{equation}
        \begin{split}
            \partial_{y^i} \Phi_r(y) &= \ps{e_i}{y-m_y} - \ps{\nabla m_y\qua{e_i}}{y-m_y} \\
            & = \ps{\hat{\pi}_y^\perp\qua{e_i}}{y-m_y} - \ps{(\hat{\pi}_y -\nabla m_y)\qua{e_i}}{y-m_y}
        \end{split}
    \end{equation}
    and thus
    \begin{equation}
        \begin{split}
            \abs{\nabla \Phi_r(y)}^2-\underbrace{\abs{y-m_y}^2}_{=\frac 1 2 \Phi_r(y)}
            = &\sum_i (\ps{(\hat{\pi}_y -\nabla m_y)\qua{e_i}}{y-m_y})^2+\\
            +& \sum_i\qua{2\ps{\hat{\pi}_y^\perp\qua{e_i}}{y-m_y}\ps{(\hat{\pi}_y -\nabla m_y)\qua{e_i}}{y-m_y}}\\
            \leq &C(n)\sum_i\abs{\ton{y-m_y}_i}^2 \ton{\abs{\ton{\hat{\pi}_y -\nabla m_y}[e_i]}^2 + 2\abs{\ton{\hat{\pi}_y -\nabla m_y}\qua{e_i}}}\\
            \leq &C(n)\delta \abs{y-m_y}^2 =C(n)\delta\Phi_r(y)
        \end{split}
    \end{equation}
    by the Subspace Selection Lemma. This proves \eqref{it_2bis}.
    }

    \par{
    We now focus on \eqref{it_3}. For $i,j$ we have as a matrix
    \begin{equation}
        \ton{\hat{\pi}^\perp_y}_{i,j} = \ps{\hat{\pi}_y^\perp[e_j]}{e_i} = \ps{\hat{\pi}_y^\perp[e_j]}{\hat{\pi}_y^\perp[e_i]}. 
    \end{equation}
    Moreover, by direct computation we have
    \begin{equation}
        \begin{split}
            &\partial_{y^i}\partial_{y^j}\Phi_r(y) = \partial_{y^i} \ps{\hat{\pi}_y^\perp[e_j]}{y-m_y} + \partial_{y^i} \ps{\ton{\hat{\pi}_y - \nabla m_y}[e_j]}{y-m_y}\\
            & = \partial_{y^i}\ps{e_j}{y-m_y} + \partial_{y^i} \ps{\ton{\hat{\pi}_y- \nabla m_y}[e_j]}{y-m_y}.
        \end{split}
    \end{equation}
    We expand and estimate the terms one by one. We have
    \begin{equation}
        \begin{split}
            &\partial_{y^i}\ps{e_j}{y-m_y} = \ps{e_j}{\hat{\pi}_y^\perp[e_i]} + \ps{e_j}{\ton{\hat{\pi}_y - \nabla m_y}[e_i]}\\
            & = \ps{\hat{\pi}_y^\perp[e_j]}{\hat{\pi}_y^\perp[e_i]} + \ps{e_j}{\ton{\hat{\pi}_y - \nabla m_y}[e_i]}. 
        \end{split}
    \end{equation}
    Then, 
    \begin{equation}
        \begin{split}
            &\abs{\partial_{y^i}\ps{e_j}{y-m_y} - \ps{\hat{\pi}_y^\perp[e_j]}{\hat{\pi}_y^\perp[e_i]}}=\abs{\ps{e_j}{\ton{\hat{\pi}_y - \nabla m_y}[e_i]}}\\
            & \leq \abs{e_j}\abs{\hat{\pi}_y - \nabla m_y}\leq C(n)\delta. 
        \end{split}
    \end{equation}
    Moreover, 
    \begin{equation}
        \begin{split}
            &\abs{\partial_{y^i}\ps{\ton{\pi_y - \nabla m_y}[e_j]}{y-m_y}} \\
            & \leq \abs{\ps{\partial_{y^i}\partial_{y^j}(y-m_y)}{y-m_y}} + \abs{\ps{\partial_{y^i}(y-m_y)}{\partial_{y^j}(y-m_y)}}
        \end{split}
    \end{equation}
    and we conclude by Cauchy-Schwarz, the estimates in the Subspace Selection Lemma and the fact that $\abs{y-m_y}\leq C(n) r$. 
    }
\end{proof}
}

With this last result we are ready to prove Lemma  \ref{lemma_covering} fully, leveraging the regularities of both the projection $\pi_y$ and $\Phi_r$.

\begin{proof}[Proof of Lemma \ref{lemma_covering}] 

\textit{(1)} 
Fix a center \(x_i \in C\), and let \(m_{x_i} := \pi_{x_i}(x_i)\). By Corollary~\ref{lemma_subsel_planes}, we have
\begin{gather}
|x_i - m_{x_i}| = \operatorname{dist}(x_i, L_{x_i}) \leq C(n)\delta \overline r_{x_i}.
\end{gather}

Now apply Lemma~\ref{lemma:phi}\,(1) with \(x = x_i\), \(\ell = m_{x_i}\); this gives a unique point \(y_i \in \hat{L}^\perp_{x_i} + m_{x_i}\) such that \(\Phi_r(y_i) = 0\), i.e., \(y_i \in S_r\). By Lemma~\ref{lemma:phi}\,(1), we also have 
\begin{gather}
|y_i - m_{x_i}| \leq C(n)\delta \overline{r}_{m_{x_i}}.
\end{gather}
Since \(\operatorname{Lip}(\overline{r}_y)\leq 500\)  (see definition of $\overline{r}_y$ and \ref{rthm} point (3)), it follows that
\begin{gather}
|\overline{r}_{m_{x_i}} - \overline{r}_{x_i}| \leq 500\,|x_i - m_{x_i}| \leq C(n)\delta \overline r_{x_i},
\end{gather}
so: 
\begin{gather}
\overline{r}_{m_{x_i}} \leq \overline{r}_{x_i} + C(n)\delta \overline r_{x_i} \leq C(n) \overline{r}_{x_i}.
\end{gather}
Hence,
\begin{gather}
|x_i - y_i| \leq |x_i - m_{x_i}| + |m_{x_i} - y_i| \leq C(n)\delta \overline{r}_{x_i}.
\end{gather}
We can conclude $x_i \in B_{\varepsilon r_{x_i}}(S_r)$ as we can require $\delta(\varepsilon)$ to satisfy $2C(n)\delta \leq \varepsilon$. This holds for all \(x_i \in C\), and in particular for every center of a bad ball.

\bigskip

\textit{(2)} Let \( y \in B_{\frac{r_{x_i}}{10}}(x_i) \cap S_r \). This intersection is nonempty because we have shown that \( x_i \in B_{\varepsilon r_{x_i}}(S_r) \). We can apply estimate (3) from Lemma \ref{rthm} which allows us to state that:
\begin{gather}
r_y \geq \frac{r_{x_i}}{2} \quad \text{and} \quad \tilde{r}_y = \frac{r_y}{100}\geq \frac{r_{x_i}}{200} \quad \forall r \leq \frac{r_{x_i}}{2}.
\end{gather}
As the radii $\tilde r_y$ don't change below this choice of the parameter $r$ what we obtain for \(y \in B_{\frac{r_{x_i}}{20}}(x_i) \cap S_r \) is that the partition of unit in Lemma \ref{partition} remains unaltered.  Moreover, as we use the partition of unity to associate the points to the planes in the Subspace Selection Lemma  \ref{lemma_subsel} using:
\begin{gather}
M_y = \sum_\alpha \phi_\alpha(y)\hat \pi_\alpha, \quad \text{and} \quad \ell_y = \sum_\alpha \phi_\alpha(y)\pi_\alpha[y]\, ,
\end{gather}
we also obtain invariance of the plane associated to each point.
But this implies that all the points $y \in B_{\frac{r_{x_i}}{20}}(x_i) \cap S_r$ will also preserve their distance from the such planes, and since $y \in S_r$ it will remain zero. Given the definition of $S_r = \Phi^{-1}(0)$ we also see that the points satisfying this property don't change for $r \leq \frac{r_{x_i}}{2}$. 

\bigskip

\textit{(3)} This can be proven by the fact that every bad ball is obtained by a Vitali covering, therefore we already know that $\B {r_i/5}{x_i}\cap \B {r_j/5}{x_j}=\emptyset$ if $i\neq j$.

\bigskip

\textit{(4)} The first thing that we need to do is to explicitly address the dependence of the manifold $S_r$ on the parameter $r_0 > 0$, with the end goal of defining the set $S_0 = \lim_{r \rightarrow 0} S_r$. 

We can start by considering a point $x \in S \cap B_{r_{x_i}}(x_i)$ where $B_{r_{x_i}}(x_i)$ is a good ball (see Definition \ref{def:bad-balls}). Using the Lipschitz estimate point (3) of Theorem \ref{rthm} on the radius we can conclude that $r_x \leq 6r_{x_i} = 6r_0$. This estimate allows us to conclude that for every point inside a good ball it holds:
\begin{gather}\label{pcon}
d_{\mathcal{H}}(S \cap B_{10\overline{r}_{x}}(x), L_{x} \cap B_{10\overline{r}_{x}}(x)) < C(n)\delta\, \overline{r}_{x_i} = C(n)\delta r \quad \forall r > r_0,
\end{gather}
and thus, using point (1) of this proof it will also hold for every $y \in S_r \cap B_{r_{x_i}}(x_i)$ as $S$ as $\varepsilon  << 1$. Using inequality \ref{pcon} we can state that on balls $B_{r_{x_i}}(x_i)$ which satisfy the good condition the following hold 
\begin{equation}\label{estimate}
    d_\cH(S \cap B_{r_{x_i}}(x_i),S_r \cap B_{r_{x_i}}(x_i)) \leq C(n)\delta r \quad \forall r > r_0
\end{equation}
and 
\begin{gather}\label{cauchy}
    d_\cH(S_r \cap B_{r_{x_i}}(x_i),S_{\frac{r}{2}} \cap B_{r_{x_i}}(x_i)) \leq C(n)\delta r \quad \forall r > r_0
\end{gather}
where \ref{cauchy} is obtained using triangle inequality.

By letting $r \rightarrow r_0$ and $r_0 \rightarrow 0$, inequalities \ref{estimate} show that both the fact that $S_r$ is a Cauchy sequence (by fixing discrete values of $r$) and that it is converging to $S$ in the Hausdorff topology, implying the thesis.

\bigskip

\textit{(5)} Estimate (2) from Lemma~\ref{lemma:phi}, applied with $\delta$ sufficiently small, guarantees that for every $x \in S_r$, we have
\begin{gather}
\Phi_r(x) = 0 \quad \text{and} \quad \nabla \Phi_r(x) = 0.
\end{gather}
Estimate (3) from Lemma~\ref{lemma:phi} provides control on the second derivatives of $\Phi_r$. Specifically, we have
\begin{gather}
\left| \nabla^2 \Phi_r(x) - \hat\pi_y^\perp \right| \leq C(n)\delta,
\end{gather}
which shows that the Hessian with respect to the normal directions is close to a projection operator. In particular, for $\delta$ sufficiently small, $\nabla^2\Phi_r$ is invertible. 
We can now apply the Implicit Function Theorem. The vanishing of both $\Phi_r$ and $\nabla \Phi_r$ on $S_r$, together with the invertibility of the Hessian in the normal directions, implies that there exists an open set $U \subseteq \mathbb{R}^k$ and a Lipschitz function
\begin{gather}
g: U \to \mathbb{R}^{n-k}
\end{gather}
such that for all $x_k \in U$,
\begin{gather}
\Phi_r(x, g(x)) = 0 \quad \text{and} \quad \nabla \Phi_r(x, g(x)) = 0.
\end{gather}
That is, $S_r$ coincides locally with the graph of $g$ over $L_x$.

Moreover, the Lipschitz constant of the graph function $g$ can be estimated using the block structure of the Hessian of $\Phi_r$. Thus, by the Implicit Function Theorem, we can estimate
\begin{equation}
\label{eq:Lipg}
\| \nabla g \|
\le C(n)\delta\, .
\end{equation}

\bigskip

\textit{(6)} By Lemma \ref{lemma:phi} (2), for sufficiently small $\delta$, the critical and zero sets of $\Phi$ coincide. Thus, for every $x \in S_r$, we have $T_x(S_r) = L_x$, the plane defined in the Subspace Selection Lemma. Corollary \ref{lemma_subsel_planes} implies the following closeness estimate:
\begin{gather}
d_\cH(S \cap B_{10\overline{r}_x}(x), L_x \cap B_{10\overline{r}_x}(x)) < C(n)\delta \overline{r}_x\, .
\end{gather}
Moreover, since $S \cap B_{\overline{r}_x}(x) \subset \B{\varepsilon \overline{r}_x}{L_{x,\overline{r}_x}}$, we have
\begin{gather}
d_\cH(L_{x,\overline{r}_x} \cap B_{\overline{r}_x}(x), L_x \cap B_{\overline{r}_x}(x)) < C(n)\delta \, .
\end{gather}
These bounds are valid only down to scale $\overline{r}_x$, where the balls remain "good"; below that scale, no controlled bound is guaranteed.

Using this control and the smoothness of the calibration, we deduce for $r \geq {r}_x$:
\begin{gather}
\Omega[L_x] \geq \Omega[L_{x,r}] - C(n)\delta \, .
\end{gather}

As the parameter $\varepsilon$ is arbitrary, we can ask the following inequality to hold $4 C(n)\delta \leq 2\varepsilon \leq \alpha$ so that we can write the chain on inequalities
\begin{equation}
    \Omega[L_{x,r}] - C(n)\delta \geq \Omega[L_{x,r}] - \frac{\varepsilon}{2} \geq \frac{3 \alpha}{4}\, .
\end{equation}
Thus, the calibration remains uniformly positive for planes $T_x(S_r \cap B_{r_{x_i}}(x_i))$ for some good ball $B_{r_{x_i}}(x_i)$. The analogous bound for $\eta$-calibration is obtained using:
\begin{gather}
\Omega[L_x] \geq \Omega_0[L_x] - \eta\, .
\end{gather}
and again, as $\eta$ is an arbitrary parameter, the following holds for every $2\eta \leq \alpha$
\begin{gather}
    \Omega_0[L_x] \geq \frac{\alpha}{4}\, .
\end{gather}

\bigskip

\textit{(7)} By Definition \ref{def:bad}, for a small enough $\varepsilon$ the statement implies that $B_{r_{x_i}}(x_i)$ is a bad ball, meaning that it doesn't contain a set of vectors which are $\varepsilon$-linearly independent. From this follows that the whole set $S$ inside this ball is contained in a tubular neighborhood of a $k-1$ dimensional subspace in the following way $S \subset B_{\varepsilon r_{x_i}}(V_{x_i})$ getting the thesis.

\end{proof}

\subsection{Proof of Theorem 2.3}

In this section, we prove the main theorem. By leveraging the regularity properties of the approximating manifolds \( S_r \), we obtain the measure bound stated in Theorem~\ref{me} through a recursive covering argument. The core idea is to partition the domain into \emph{good regions}, where a measure estimate is already valid, and \emph{bad regions}, to which we reapply the argument inductively. The initial estimate on the good regions can be obtained from known results in the literature.

More precisely, we apply the main result from \cite[Theorem~1.4]{ENV_calibrated_Reifenberg} to each approximating manifold \( S_r \). Since each \( S_r \) satisfies the almost calibration property with the same uniform constant \( A = A(n,\alpha,\varepsilon) > 1 \), we have
\begin{gather}
\cH^k(S_r \cap B_{r_2}(x)) \leq A w_k r_2^k \quad \text{for all } r, r_2 > 0 \text{ and } x \in B_2(0).
\end{gather}
Using the lower semicontinuity of the Hausdorff measure for the sequence $S_r$, along with point (4) of Lemma~\ref{lemma_covering}, we can pass this estimate to the limit set
\begin{gather}
\tilde{S}_0 := S_0 \setminus \left( \bigcup_{i\in I_b} B_{r_{x_i}}(x_i) \right),
\end{gather}
and conclude that
\begin{gather}\label{eq_HktildeS0}
 \cH^k(\tilde S_0 \cap B_{r_2}(x)) \leq \cH^k(S_0 \cap B_{r_2}(x)) \leq \liminf_{r \to 0} \cH^k(S_r \cap B_{r_2}(x)).
\end{gather}
Moreover, given that $\cur{\B{r_i/5}{x_i}}_{i\in I_b}$ are pairwise disjoint and $S_0$ is a Lipschitz graph over these sets, we have that
\begin{gather}
 \sum_{i\in I_b} r_i^k\leq C(n)\cH^k(S_0 \cap B_{r_2}(x))
\end{gather}

\vspace{5mm}

Note that \eqref{eq_HktildeS0} does not include information on $S$ on bad balls. In particular,
\begin{gather}
 S=\tilde S_0 \cup \bigcup_{i\in I_b} \ton{S\cap \B {r_i}{x_i}}\, .
\end{gather}
However, by point \eqref{it_Vbad} in Lemma~\ref{lemma_covering}, for each $i\in I_b$ there exists a $k-1$ dimensional $V_i$ such that:
\begin{gather}
 S\cap \B {r_i}{x_i}\subset \B{\varepsilon r_i}{V_i}\, .
\end{gather}
Given that $V_i$ is $k-1$-dimensional, this will allow us to start over the covering procedure over the bad balls in a way that preserves $k$-dimensional measure bounds.

\vspace{5mm}
In the following proof we specify all the details.
\begin{proof}

We will prove that \(S \cap B_1(0)\) is \(k\)-rectifiable and satisfies
\begin{gather}
\cH^k(S \cap B_r(x)) \le A(n,\alpha,\varepsilon)\,\omega_k\,r^k
\quad \text{whenever } B_{2r}(x) \subset B_2(0).
\end{gather}
By translating and scaling, it suffices to show
\begin{gather}
\cH^k(S \cap B_1(0)) \le A(n,\alpha,\varepsilon)\,\omega_k.
\end{gather}

The proof is basically an inductive application of Lemma~\ref{lemma_covering}. We start by observing that, for every $r$, $S_r$ is an oriented $k-$dimensional manifold in $B_{1}(0)$ with $S_r = L_{0,1}$ outside of $B_{2}$. Moreover, as we proved above we have uniform bounds
\begin{equation}
    \mathcal{H}^k\ton{S_r\cap B_{3/2}}\leq A\omega_k
\end{equation}
Thus we can apply standard results (see for example \cite[Theorem 1.5]{Ambrosio_Currents}) to conclude that taken a sequence $r_i\to 0$, up to passing to a subsequence we have $S_{r_i}\to S_0$ in the sense of integral currents. In particular, as a set $S_0$ is $k$-rectifiable and
\begin{equation}
    \mathcal{H}^k\ton{S_0\cap B_{3/2}}\leq A\omega_k
\end{equation}

\textbf{Inductive claim.} We claim that we can cover $S$ inductively in $j$ with sets $\tilde S_j$ and balls indexed by $s\in J_j$ in such a way that for all $j$:
\begin{gather}
 S\subset \tilde S_j \cup \bigcup_{s\in J_j} (S\cap \B {r_s}{x_s})\, ,\\
 \cH^k(\tilde S_j)\leq A_0(n,\alpha,\varepsilon) \sum_{i=0}^{j}\ton{\frac 1 {10}}^i \, , \qquad \sum_{s\in J_j} r_s^k \leq \ton{\frac 1 {10}}^{j+1} A_0(n,\alpha,\varepsilon)
\end{gather}
and $\tilde S_j$ is $k$-rectifiable.

\textbf{First application of covering lemma.} In order to prove the inductive claim, we will apply the covering Lemma~\ref{lemma_covering} to \(S \cap B_1(0)\), and the definition of bad balls. In particular, by Lemma~\ref{lemma_covering} we obtain a decomposition
\begin{gather}
S \cap B_1(0) \subset S_0 \cup \bigcup_{i \in I_b} B_{r_i}(x_i)\, .
\end{gather}
where $S_0$ is $k$-rectifiable and

\begin{gather}
\cH^k(S_0) \leq A_0 \quad \text{and} \quad \sum_{i\in I_b} r_i^k \leq C(n)A_0\, .
\end{gather}
In order to deal with
\begin{gather}
 \bigcup_{i \in I_b} \ton{S\cap B_{r_i}(x_i)}
\end{gather}
we recall that, by point \eqref{it_Vbad} in Lemma~\ref{lemma_covering}, for each $i\in I_b$ there exists a $k-1$ dimensional $V_i$ such that:
\begin{gather}
 S\cap \B {r_i}{x_i}\subset \B{\varepsilon r_i}{V_i}\, .
\end{gather}
Given that $V_i$ is $k-1$-dimensional, we can easily cover this set by $C(n)\varepsilon^{1-k}$ balls of radius $\varepsilon r_i$, and in particular:
\begin{gather}
 \bigcup_{i\in I_b} S\cap \B {r_i}{x_i}\subset \bigcup_{i\in I_b}\bigcup_{j\in J_i} \B{\varepsilon r_i}{p_j}
\end{gather}
with the estimate
\begin{gather}
 \sum_{i\in I_b}\sum_{j\in J_i} \omega_k \ton{\varepsilon r_i}^k\leq C(n)\varepsilon \sum_{i\in I_b} r_i^k\leq C(n)\varepsilon A_0\, .
\end{gather}
If we choose $\varepsilon$ sufficiently small so that $C(n)\varepsilon\leq \frac 1 {10}$, we obtain the proof of the induction claim for $j=0$, with $\tilde S_0=S_0$.

By applying again Lemma~\ref{lemma_covering} to each of the balls $\B {\varepsilon r_i}{p_j}$, we obtain the proof of all the induction steps.

\end{proof}

\section{Appendix}
Here we gather some technical results used in the article that are relatively easy, albeit a bit tricky, to prove.

\subsection{Explicit computation IFT}\label{sec_IFT} We now include the full calculations relative to the bounds on the function $F$ of the proof \ref{lemma_subsel}.

\begin{lemma}
    Given $F(y,v)$ defined in equation \ref{F} and $\hat L_y$ defined at \ref{hatpi}. We show that the projection $\hat \pi_y$ relative to the subspace $\hat L_y$ is a stationary point of $F$ with respect to $\partial_w$ and the following inequality holds:
    \begin{gather}
        \frac{1}{2} \partial^2_w tr_{L_v}(M_y) =  - \langle e_w, M_y[e_w]\rangle + \langle u_j, M_y[u_j]\rangle.
    \end{gather}
\end{lemma}

\begin{proof}
Let \( w \neq 0 \) and a linear application $v: \hat L_\beta \to \hat L_\beta^\perp $ with $\|v\|\leq 0.1$. Moreover, we define
\begin{gather}
  e_w := \frac{(w, v(w))}{|(w, v(w))|}    
\end{gather}
so that \( e_w \in \hat L_v \) is a unit vector in the direction determined by \( w \). Complete \( e_w \) to an orthonormal basis \( \{e_w, e_2, \ldots, e_k\} \) of the \( k \)-dimensional subspace \( \hat L_v \subset \mathbb{R}^n \). Likewise, let \( \{u_{k+1}, \ldots, u_n\} \) be an orthonormal basis for \( L_v^\perp \), so that \( \{e_w, e_2, \ldots, e_k, u_{k+1}, \ldots, u_n\} \) forms an orthonormal basis for \( \mathbb{R}^n \).

We consider rotations of the plane \( \hat L_v \) within each 2-dimensional plane \( \operatorname{span}\{e_w, u_j\} \), for \( j = k+1, \ldots, n \), leaving all other basis directions fixed. This rotation is given by
\begin{gather}
    R_\theta = \exp(\theta A), \quad \text{where} \quad A = e_w u_j^T - u_j e_w^T.  
\end{gather}
Then \( R_\theta[e_w] = \cos\theta\, e_w + \sin\theta\, u_j \), and the rotated subspace \( L_v(\theta) \) has orthonormal basis \begin{gather}
\{R_\theta[e_w], e_2, \ldots, e_k\}.
\end{gather} 
We compute
\begin{gather}
  \operatorname{tr}_{L_v(\theta)}(M_y) = \langle R_\theta[e_w], M_y[R_\theta[e_w]]\rangle + \sum_{i=2}^{k} \langle e_i, M_y[e_i]\rangle.
\end{gather}
Expanding the first term gives
\begin{gather}
  \langle R_\theta[e_w], M_y[R_\theta[e_w]] \rangle = \cos^2\theta \langle e_w, M_y[e_w] \rangle + \sin^2\theta \langle u_j, M_y[u_j] \rangle + \sin(2\theta) \langle e_w, M_y[u_j] \rangle.
\end{gather}
Thus, we obtain:
\begin{gather}
  \operatorname{tr}_{L_v(\theta)}(M_y) = \cos^2\theta \langle e_w, M_y[e_w] \rangle + \sin^2\theta \langle u_j, M_y[u_j] \rangle + \sin(2\theta) \langle e_w, M_y[u_j] \rangle + \sum_{i=2}^{k} \langle e_i, M_y[e_i] \rangle.
\end{gather}
Taking derivatives with respect to \( \theta \), we find
\begin{gather}
  \partial_w \operatorname{tr}_{L_v}(M_y) = \left. \frac{d}{d\theta} \operatorname{tr}_{L_v(\theta)}(M_y) \right|_{\theta=0} = 2 \langle e_w, M_y[u_j] \rangle.
\end{gather}
Now, if \( e_w \in \operatorname{span}\{v_1, \ldots, v_k\} \), the span of the top \( k \) eigenvectors of \( M_y \), then
\begin{gather}
     M_y[e_w] = \sum_{i=1}^k \lambda_i \langle e_w, v_i \rangle v_i, 
\end{gather}
which lies entirely in \( L_v \), hence orthogonal to \( u_j \in L_v^\perp \). Therefore,
\[
  \langle M_y[e_w], u_j \rangle = 0,
\]
and so \( \partial_w \operatorname{tr}_{L_v}(M_y) = 0 \) for all directions \( u_j \in L_v^\perp \). This implies the gradient of the trace functional in the graph coordinates vanishes: \( F(y, \hat \pi_y) = 0 \).

To estimate the Hessian, we differentiate again:
\[
  \frac{1}{2} \partial_w^2 \operatorname{tr}_{L_v}(M_y) = \left. \frac{d^2}{d\theta^2} \operatorname{tr}_{L_v(\theta)}(M_y) \right|_{\theta=0} = -\langle e_w, M_y[e_w] \rangle + \langle u_j, M_y[u_j] \rangle,
\]
which completes the computation.

\end{proof}

\subsection{Proof of estimates \texorpdfstring{\eqref{eq_claim1} and \eqref{eq_claim2}}{in Claim 1 and Claim 2}}\label{Claims} We now address the proof of Claim $1$ and $2$ in Theorem \ref{lemma_subsel_m}, showing controlled first and second order derivatives of both $\ell_y$ and $M_y$.

\begin{theorem}
Given $\ell_y$ and $M_y$ as in \ref{lemma_subsel}, the following inequalities hold: 
\begin{equation}
    |\partial_{y^i} (\ell_y- \hat{\pi}_y[y])| \leq C(n)\delta, \quad |\partial_{y^i} \partial_{y^j} \ell_y| \leq   C(n)\delta \tilde{r}_y^{-1},
\end{equation}
\begin{equation}
    \left| \partial_{y^i} (M_y[y]-\hat{\pi}_{\beta}[y])\right| \leq C(n)\delta, \quad \left|\partial_{y^i} \partial_{y^j} M_y[y]\right| \leq C(n)\delta \tilde{r}_y^{-1}.
\end{equation} 
\end{theorem}

\begin{proof}

We compute explicitly:
\begin{equation}
    |\partial_{y^i} (\ell_y- \hat{\pi}_\beta[y])| = |\partial_{y^i}\ell_y - \hat{\pi}_\beta[e_i]|, 
\end{equation}
starting from
\begin{equation}\label{partialell}
    \partial_{y^i} \ell_y = \sum \partial_{y^i} \phi_{\alpha}(y) \pi_\alpha[y] + \sum  \phi_{\alpha}(y) \partial_{y^i} \pi_\alpha[y] = \sum \partial_{y^i} \phi_{\alpha}(y) \pi_\alpha[y] + \sum  \phi_{\alpha}(y) \hat{\pi}_\alpha[e_i].
\end{equation}
Since \( \sum_{\alpha} \phi_{\alpha} = 1 \) on \( B_2 \), it also holds that \( \sum_{\alpha} \partial_{y^i} \phi_{\alpha} = 0 \) on the same domain and thus, using \eqref{partialell} we obtain: 
\begin{equation}
\begin{split}
    \left|\sum \partial_{y^i} \phi_{\alpha}(y) \pi_\alpha[y] + \sum  \phi_{\alpha}(y) \hat{\pi}_\alpha[e_i] - \hat{\pi}_\beta[e_i]\right| \\
    & = \left|\sum \partial_{y^i} \phi_{\alpha}(y) \pi_\alpha[y] + M_y[e_i] - \hat{\pi}_\beta[e_i]\right| \\ 
    &\leq \left|\sum \partial_{y^i} \phi_{\alpha}(y) \pi_\alpha[y]\right| + \left|M_y[e_i] - \hat{\pi}_\beta[e_i]\right| \\
    & \leq \left|\sum \partial_{y^i} \phi_{\alpha}(y) \pi_\alpha[y]\right| + C(n)\delta,
\end{split}
\end{equation}
where the last step is justified by the estimate on the operator norm of the difference of the matrices, moreover the first term can be bounded in the following way:
\begin{equation}
    \left|\sum_\alpha \partial_{y^i} \phi_{\alpha}(y) \pi_\alpha[y]\right| =  \left|\sum_\alpha \partial_{y^i} \phi_{\alpha}(y) (\pi_\alpha[y] - \pi_{\beta}[y]) \right|,
\end{equation}
and 
\begin{equation}\label{l1}
\begin{split}
   \left|\sum_\alpha \partial_{y^i} \phi_{\alpha}(y) (\pi_\alpha[y] - \pi_{\beta}[y]) \right| \\  
 & \leq \sum_\alpha |\partial_{y^i} \phi_{\alpha}(y)(\pi_\alpha[y] - \pi_{\beta}[y])| \\
 & \leq \sum_\alpha |\partial_{y^i} \phi_{\alpha}(y)| |(\pi_\alpha[y] - \pi_{\beta}[y])|  \\
 & \stackrel{\text{Lemma }\ref{lemma_planes}}{\leq} C(n)\delta
\end{split}
\end{equation}
The second part of the first claim can be proved in a similar way:
\begin{equation}
    \partial_{y^i} \partial_{y^j} \ell_y = \sum_\alpha \partial_{y^i} \partial_{y^j}\phi_\alpha(y) \pi_{\alpha}[y] +
    \sum_\alpha \partial_{y^i} \phi_\alpha(y) \hat{\pi}_\alpha[e_j] + \sum_\alpha \partial_{y^j} \phi_\alpha(y) \hat{\pi}_\alpha[e_i]. 
\end{equation}
and now
\begin{equation}
    \begin{split}
        \left|\sum_\alpha \partial_{y^i} \partial_{y^j}\phi_\alpha(y) (\pi_{\alpha}[y]-\pi_{\beta}[y]) \right| \leq \sum_\alpha |\partial_{y^i} \partial_{y^j}\phi_\alpha(y)| C(n)\delta \tilde{r}_y\leq C(n)\delta \tilde{r}_y^{-1}
    \end{split}
\end{equation}
where the last inequality is justified by point \textit{(4)} of Lemma \ref{partition}. Moreover the other terms can be bounded from above in the following way:
\begin{equation}
     \left|\sum_\alpha \partial_{y^i} \phi_\alpha(y) \hat{\pi}_\alpha[e_j]\right| = \left|\sum_\alpha \partial_{y^i} \phi_\alpha(y) (\hat{\pi}_\alpha[e_j] - \hat{\pi}_\beta[e_j])\right| \leq C(n) \delta \tilde{r}_y^{-1}.
\end{equation}
This last two inequalities allow us to conclude the estimate:
\begin{equation}\label{l2}
 |\partial_{y^i} \partial_{y^j} \ell_y| \leq   C(n)\delta \tilde{r}_y^{-1}
\end{equation}
finishing the proof of the first claim. 

We now address the second claim, as for the first one we compute explicitly:
\begin{equation}
   \left| \partial_{y^i} (M_y[y]-\hat{\pi}_{\beta}[y])\right| = \left|\sum_\alpha \partial_{y^i}\phi_\alpha(y) \hat{\pi}_{\alpha}[y] +
    \sum_\alpha \phi_\alpha(y) (\hat{\pi}_\alpha[e_i] - \hat{\pi}_\beta [e_i])\right|.
\end{equation}

Applying triangle inequality we can bound each term separately, starting from the second:
\begin{equation}\label{M1}
\begin{split}
       \left|\sum_\alpha \phi_\alpha(y) \hat{\pi}_\alpha[e_i] - \hat{\pi}_\beta [e_i]\right| \\ 
       & \leq \left|\sum_\alpha \phi_\alpha(y) \hat{\pi}_\alpha[e_i] - \hat{\pi}_\beta [e_i]\right| \\
       & = \left| M_y[e_i] - \hat{\pi}_\beta [e_i] \right| \leq C(n)\delta.
\end{split}
\end{equation}
Similarly we also obtain the estimate on the second derivative of $M_y$
\begin{equation}
    \partial_{y^i} \partial_{y^j} M_y[y] = \sum_\alpha \partial_{y^i} \partial_{y^j}\phi_\alpha(y) \hat{\pi}_{\alpha}[y] +
    \sum_\alpha \partial_{y^i} \phi_\alpha(y) \hat{\pi}_\alpha[e_j] + \sum_\alpha \partial_{y^j} \phi_\alpha(y) \hat{\pi}_\alpha[e_i]  
\end{equation}
with the same bounds:
\begin{equation}
    \left|\sum_\alpha \partial_{y^i} \phi_\alpha(y) \hat{\pi}_\alpha[e_j]\right| = \left|\sum_\alpha \partial_{y^i} \phi_\alpha(y) (\hat{\pi}_\alpha[e_j]-\hat{\pi}_\beta[e_j])\right| \leq C(n)\delta 
\end{equation}
and 
\begin{equation}
    \left|\sum_\alpha \partial_{y^i} \partial_{y^j}\phi_\alpha(y) \hat{\pi}_{\alpha}[y]\right| = \left|\sum_\alpha \partial_{y^i} \partial_{y^j}\phi_\alpha(y) (\hat{\pi}_{\alpha}[y]- \hat{\pi}_{\beta}[y])\right| \leq C(n)\delta \tilde{r}_y^{-1}, 
\end{equation}
lead to $\left|\partial_{y^i} \partial_{y^j} M_y[y]\right| \leq C(n)\delta \tilde{r}_y^{-1}$ and thus
\begin{equation}\label{M2}
     \left|\partial_{y^i} \partial_{y^j} M_y[y]\right| \leq C(n)\delta \tilde{r}_y^{-1}.
\end{equation}

\end{proof}

\bibliographystyle{aomalpha}
\bibliography{CaliReif}

\end{document}